\documentclass[11pt]{article}
\usepackage[utf8]{inputenc}

\usepackage{epsfig,epsf,fancybox}
\usepackage{amsmath}
\usepackage{mathrsfs}
\usepackage{amssymb}
\usepackage{graphicx}
\usepackage{color}
\usepackage{multirow}
\usepackage{paralist}
\usepackage{verbatim}
\usepackage{galois}
\usepackage{algorithm}
\usepackage{algorithmic}
\usepackage{boxedminipage}
\usepackage{booktabs}
\usepackage{accents}
\usepackage{stmaryrd}
\usepackage{subfig}
\usepackage{epstopdf}
\usepackage{appendix}
\usepackage{amsthm}
\usepackage{cases}  
\usepackage[top=1in,bottom=1.2in,left=1in,right=1in,xetex]{geometry}

\newtheorem{theorem}{Theorem}[section]
\newtheorem{corollary}[theorem]{Corollary}
\newtheorem{lemma}[theorem]{Lemma}

\newtheorem{definition}[theorem]{Definition}
\newtheorem{example}{Example}[section]

\newtheorem{remark}[theorem]{Remark}

\def\endproof{\hfill $\Box$ \vskip 0.4cm}

\newcommand{\f}{\mathscr{F}}
\newcommand{\lr}{\mathcal{L}}
\newcommand{\e}{\mathbb{E}}
\newcommand{\br}{\mathbb{R}}
\newcommand{\pr}{\mathcal{P}}

\newcommand{\p}{\mathcal{P}}
\newcommand{\dd}{\mathcal{D}_{C_1}}
\newcommand{\argmin}{\mathop{\rm argmin}}
\newcommand{\tr}{\textnormal{Tr}\,}

\newcommand{\hr}{\mathcal{H}}

\allowdisplaybreaks[4]

\title{Fully Coupled Nonlocal Quasilinear Forward-Backward Parabolic Equations Arising from Mean Field Games}
\author{Ziyu Huang\footnotemark[1] \and Shanjian Tang\footnotemark[2]} 
\date{}

\begin{document}	

\maketitle

\renewcommand{\thefootnote}{\fnsymbol{footnote}}

\footnotetext[1]{School of Mathematical Sciences, Fudan University, Shanghai 200433, China. E-mail: zyhuang19@fudan.edu.cn.}
\footnotetext[2]{Department of Finance and Control Sciences, School of Mathematical Sciences, Fudan University, Shanghai 200433, China. E-mail: sjtang@fudan.edu.cn. Research partially supported by
	National Natural Science Foundation of China (Grants No. 11631004 and No. 12031009).}

\begin{abstract}
In this paper, we study fully coupled nonlocal second order quasilinear forward-backward partial differential equations (FBPDEs), which arise from solution of the mean field game (MFG) suggested by Lasry and Lions [Japan. J. Math. 2 (2007), p. 237 (Remark iv)]. We show the existence of solutions $(u,m)\in C^{1+\frac{1}{4},2+\frac{1}{2}}([0,T]\times\br^n)\times C^{\frac{1}{2}}([0,T],\p_1(\br^n))$, and also the uniqueness under an additional monotonicity condition. Then, we improve the regularity of our weak solution $m$ to get a classical solution under appropriate regularity assumptions on coefficients. The FBPDEs can be used to investigate a system of mean field equations (MFEs), where the backward one is a Hamilton-Jacobi-Bellman equation and the forward one is a Fokker-Planck equation. Moreover, we prove a verification theorem and give an optimal strategy of the associated MFG via the solution of MFEs. Finally, we address the linear-quadratic problems.
\vskip 4.5mm

\noindent \begin{tabular}{@{}l@{ }p{10.1cm}} {\bf Keywords } &
	FBPDEs, mean field equation, Fokker-Planck equation, HJB equation, MFGs
\end{tabular}

\noindent {\bf 2000 MR Subject Classification } 93E20, 60H30, 35K55, 35K10

\end{abstract}
	
\section{Introduction}

In this paper, we consider the following second order forward-backward partial differential equations (FBPDEs):
\begin{equation}\label{In:3}
	\left\{
	\begin{aligned}
		&\frac{\partial u}{\partial t}(t,x)+\sum_{i,j=1}^na_{ij}(t,x,m(t,\cdot))\frac{\partial^2u}{\partial x_i\partial x_j}(t,x)\\
		&\qquad+\hr(t,x,m(t,\cdot),Du(t,x))=0,\quad (t,x)\in(0,T]\times\br^n;\\
		&\frac{\partial m}{\partial t}(t,x)-\sum_{i,j=1}^n\frac{\partial^2}{\partial x_i\partial x_j}[a_{ij}(t,x,m(t,\cdot))m(t,x)]\\
		&\qquad+div[\frac{\partial \hr}{\partial p}(t,x,m(t,\cdot),Du(t,x))m(t,x)]=0,\quad (t,x)\in[0,T)\times\br^n;\\
		&m(0,x)=m_0(x), \quad u(T,x)=g(x,m(T,\cdot)),\quad x\in\br^n.
	\end{aligned}
	\right.
\end{equation}
It is a system of fully coupled forwardly nonlocal quasilinear forward-backward parabolic equations, where
\begin{align*}
	&\sigma:[0,T]\times\br^n\times\pr(\br^n)\to\br^{n\times n},\quad a_{ij}=\frac{1}{2}(\sigma\sigma^*)_{ij},\quad 1\le i,j\le n;\\
	&\hr:[0,T]\times\br^n\times\pr(\br^n)\times\br^n\to\br;\quad g:\br^n\times\pr(\br^n)\to\br.
\end{align*}
In general, the coupling between the first and the second equation makes the system a novel one for which no existing theory or approach seems to be applicable directly---as has been noted for the special case of free  mean field  in both of the drift and diffusion $\sigma$ by Lasry and Lions~\cite[the third paragraph, p. 232]{JM3}. 

Our main result is to show the existence of solutions $(u,m)\in C^{1+\frac{1}{4},2+\frac{1}{2}}([0,T]\times\br^n)\times C^{\frac{1}{2}}([0,T],\p_1(\br^n))$ of FBPDEs \eqref{In:3} (see Theorem~\ref{thm:exist}). We also show the uniqueness under additional monotonicity conditions (see Theorem~\ref{thm:uni}). Moreover, we improve the regularity of our weak solution $m$, and give the existence result of classical solutions of FBPDEs \eqref{In:3} in $C^{1+\frac{1}{4},2+\frac{1}{2}}([0,T]\times\br^n)\times C^{1+\frac{1}{4},2+\frac{1}{2}}([0,T]\times\br^n)$ under appropriate regularity assumptions on coefficients (see Theorem~\ref{thm:improve}). For the particular case
\begin{align*}
	\hr(t,x,m,p):=H(t,x,m,\phi(t,x,m,p),p),\quad (t,x,m,p)\in[0,T]\times\br^n\times\pr(\br^n)\times\br^n,
\end{align*}
where $H$ is a Hamiltonian and $\phi(t,x,m,p)$ is the minimizing feedback control function of $(t,x,m,p)$, FBPDEs \eqref{In:3} is a system of mean field equations (MFEs), where the backward one is a Hamilton-Jacobi-Bellman (HJB) equation and the forward one is a Fokker-Planck (FP) equation with the diffusion coefficient $\sigma (t,x,m)$ varying with both of the state and the mean-field---the mean field game (MFG) setting suggested by Lasry and Lions~\cite[Remark iv, p. 237]{JM3}. We give the existence and uniqueness results for the MFEs (see Theorems~\ref{thm:mfe_exist} and \ref{thm:mfe_uni}) and show that for a solution of the MFEs, the feedback strategy is an optimal control for the associated MFG (see Theorem~\ref{thm_mfg}). And our results can also be applied to linear-quadratic cases (see Section~\ref{EXA}).

Solution of optimal control problems for either of stochastic differential equations (SDEs) or evolutionary equations in an infinite dimensional space,  appeals to a very closely system of FBPDEs. Fleming~\cite{Fleming} characterized the optimal feedback control of a stochastic control problem as the following system of FBPDEs:
\begin{equation*}\left\{\begin{aligned}
		&\frac{\partial u}{\partial t} +\max_{\alpha}\Big\{\frac{1}{2}tr[ (\sigma\sigma^*)(t,x,\alpha)D^2u]+b(t,x, \alpha)\cdot D u+f(t,x,\alpha)\Big\}=0;\\
		&\frac {\partial m}{\partial t}-\frac{1}{2}tr D^2[(\sigma\sigma^*)(t,x,\alpha(t,x))m]+div[b(t,x,\alpha(t,x))m]=0;\\
		&m(0,x)=m_0(x),\quad u(T,x)=g(x, m(T)),\quad x\in\br^n.
	\end{aligned}\right.
\end{equation*}
It is locally dependent on both unknown functions $u$ and $m$. The maximum principle for optimal control of an evolutionary equation leads to the following system of FBPDEs:
\begin{equation*}\left\{\begin{aligned}
		&du(t,\cdot)=[-A^*u(t,\cdot)-h(u(t,\cdot),m(t, \cdot))]\,dt, \quad u(T,x)=g(x, m(T));\\
		&dm(t, \cdot)=[Am(t, \cdot)+b(m(t, \cdot),u(t,\cdot))]\, dt, \quad m(0,x)=m_0(x)
	\end{aligned}\right.
\end{equation*}
when the generator $A$ is a partial differential operator. It is non-locally dependent on both unknown functions $u$ and $m$.
See J. L. Lions~\cite{LionsJL} and Yong~\cite{Yong} for details. Lasry and Lions \cite{JM1,JM2,JM3} proposed the following MFEs:
\begin{equation}\label{In:0}
	\left\{
	\begin{aligned}
		&\frac{\partial u}{\partial t}(t,x)+\Delta u(t,x)-H(t,x,Du)+F(t,x,m(t,\cdot))=0,\quad (t,x)\in[0,T)\times\br^n;\\
		&\frac{\partial m}{\partial t}(t,x)-\Delta m(t,x)- div[\frac{\partial H}{\partial p}(t,x,Du)m]=0,\quad (t,x)\in(0,T]\times\br^n;\\
		&m(0,x)=m_0(x), \quad u(T,x)=g(x,m(T,\cdot)),\quad x\in\br^n.
	\end{aligned}
	\right.
\end{equation}
They have received  much attention. See \cite{AB2,PC,CAR,POA} for the resolvability of MFEs for different sakes and in different settings. In a series of lectures given at the Collége de France \cite{PC}, Lions showed the existence and uniqueness of classical solutions of the particular second order MFEs~\eqref{In:0} with $H(t,x,p)=\frac{1}{2}|p|^2$. Porretta \cite{POA} established the existence and uniqueness of weak solutions of MFEs \eqref{In:0}. Bensoussan et al. \cite{AB2} solved the MFEs for linear quadratic MFGs for  $\sigma$ is a constant. The proof of the resolvability of non-linear FBPDEs is not trivial even when $\sigma$ is a constant. In general, a classical solution of the system is not available in the literature, even under high regularity assumptions on the coefficients. In this paper, we generalize Lions' existence and uniqueness results \cite[Theorems 3.1 and 3.6]{PC} to the more general one \eqref{In:3}. We do not need the linear quadratic condition and allow the volatility $\sigma$ to depend on the state and the distribution of the state. The appearance of the distribution variable in the volatility complicates the structure of FBPDEs \eqref{In:3}, makes it difficult to be linked to an optimal control problem for the nonlinear FP equation (as it is in Lasry and Lions~\cite[Subsection 2.6, pp. 249-250]{JM3}), and also increases the difficulty in the solution.  The reader is referred to  a few facts collected by  Porretta \cite[p. 3]{POA}, concerning, independently, FP and viscous Hamilton–Jacobi equations,  so as to understand properly the relevance and the difficulties related to the question of the uniqueness and regularity of solutions to FBPDEs \eqref{In:3}.
We first show a weak solution of FBPDEs \eqref{In:3} in $C^{1+\frac{1}{4},2+\frac{1}{2}}([0,T]\times\br^n)\times C^{\frac{1}{2}}([0,T],\p_1(\br^n))$ and then improve the regularity of the weak solution $m$ to get a classical solution under further regularity assumption on the coefficients. The solubility result seems to be new. The proof of the existence result is based on the Schauder's  fixed point theorem and relies on the theory of quasilinear parabolic equations~\cite{LO}. We also generalizes Lions' result \cite[Lemma 3.7, p.15]{PC} of MFGs and apply our result to linear-quadratic problems. 

The FP equations for McKean-Vlasov SDEs can be traced back to McKean \cite{McK} and were subsequently studied in more general settings \cite{JB,AS}. Barbu and Röckner \cite{BAR} considered the FP equations for the McKean-Vlasov SDEs for the case of Nemytskii-type coefficients. Tse \cite{TSE} considered the higher order regularity of nonlinear FP equations. Huang and Tang \cite{OUR1} studied the regularity of nonlinear FP equations in $L^2$ and $L^1$ spaces. In this paper, our existence result of FBPDEs \eqref{In:3} appeals to the existence result of the FP equation for McKean-Vlasov SDEs with mean-field-dependent volatility (see Remark~\ref{rk:FP}).

FBPDEs \eqref{In:3} can be decoupled with the help of the so-called ``master equation", which is a second-order PDE stated on the space of probability measures:
\begin{equation}\label{master}
	\left\{
	\begin{aligned}
		&\frac{\partial U}{\partial t}(t,x,\mu)+\sum_{i,j=1}^na_{ij}(t,x,\mu)\frac{\partial^2U}{\partial x_i\partial x_j}(t,x,\mu)+\hr(t,x,\mu,DU(t,x,\mu))\\
		&\qquad +\int_{\br^n}\Big[ \sum_{i,j=1}^na_{ij}(t,\xi,\mu)\frac{\partial^2}{\partial\xi_i\partial\xi_j}\frac{\partial U}{\partial \mu}(t,x,\mu)(\xi)\\
		&\qquad+\sum_{i=1}^n\frac{\partial\hr}{\partial p_i}(t,\xi,\mu,DU(t,\xi,\mu))\frac{\partial}{\partial\xi_i}\frac{\partial U}{\partial \mu}(t,x,\mu)(\xi)\Big]\mu(d\xi)=0,\\ &\qquad(t,x,\mu)\in[0,T]\times\br^n\times\pr(\br^n);\\
		&U(T,x,\mu)=g(x,\mu),\quad (x,\mu)\in\br^n\times\pr(\br^n).
	\end{aligned}
	\right.
\end{equation}
In fact, the solution of the last equation is associated to that of FBPDEs \eqref{In:3} in the following way
\begin{align}\label{decouple}
	u(t,x)=U(t,x,m(t,\cdot)),\quad (t,x)\in[0,T]\times\br^n.
\end{align}
The concept of master equation is introduced by P. L. Lions in his lectures at Collége de France. A PDE defined on the space of probability measures is quite different from that defined on an Euclidean space, and the solvability of such a PDE is difficult to obtain. Bensoussan et al. \cite{AB0,AB4} interpreted the master equation for both MFGs and mean field type control problems. Buckdahn et al. \cite{BR} solved the linear Kolmogorov PDE with a probabilistic method. See \cite{CAR,CHA,CD} for the solubility of master equations for different sakes and in different settings. Bensoussan et al. \cite{AB2} gave the master equations associated with stochastic MFGs, where the HJB-FP equations are forward-backward stochastic PDEs. 
Our work \cite{OUR1} discussed solubility of master equations of particular types. In Subsection~\ref{discuss}, we give a discussion on the relations between MFEs and master equation \eqref{master}. Solution of the master equation can be used to construct that of the FBPDEs \eqref{In:3}, however, it requires regularity conditions of the coefficients with respect to the distribution variable, and an analytic approach to the master equation is not available and quite appealing. So in this work, we directly study the solubility of FBPDEs \eqref{In:3}. Actually, very few works address solubility of FBPDEs\eqref{In:3}  in the literature. This work is a merely preliminary result and  many questions still remain to be further studied. It is interesting and  challenging to relax  our assumptions on the coefficients. 

The paper is organized as follows. In Section~\ref{MFE}, we give the existence and uniqueness of solutions $(u,m)\in C^{1+\frac{1}{4},2+\frac{1}{2}}([0,T]\times\br^n)\times C^{\frac{1}{2}}([0,T],\p_1(\br^n))$ of FBPDEs \eqref{In:3}, and improve the regularity of $m$ to give the existence result of classical solutions of FBPDE \eqref{In:3} in $C^{1+\frac{1}{4},2+\frac{1}{2}}([0,T]\times\br^n)\times C^{1+\frac{1}{4},2+\frac{1}{2}}([0,T]\times\br^n)$. In Section~\ref{MFG}, we show the resolvability of MFEs corresponding to MFGs and prove a verification theorem to construct an optimal strategy of the associated MFG, and we also give a discussion on the relations between MFEs and master equations. In Section~\ref{EXA}, we address the linear-quadratic problems.

\subsection{Notations}
Let $(\Omega,\f,\mathbb{P})$ denote a complete filtered probability space augmented by all the $\mathbb{P}$-null sets on which an $n$-dimensional Brownian motion $\{W_t,\ 0\le t\le T\}$ is defined. $\lr(\cdot)$ is the law. Let $\mathcal{S}^2_{\f}(0,T)$ denote the set of all $\f_t$-progressively-measurable $\br^n$-valued processes $\beta=\{\beta_t,\ 0\le t\le T\}$ such that $\e[\sup_{0\le t\le T} |\beta_t|^2]<+\infty$. Let $\pr(\br^n)$ denote the space of all Borel probability measures on $\br^n$, and $\pr_1(\br^n)$ the space of all probability measures $m\in\pr(\br^n)$ such that 
\begin{equation*}
	|m|_{\pr_1(\br^n)}:=\int_{\br^n} |x|m(dx)<\infty.
\end{equation*}
The Kantorovitch-Rubinstein distance is defined on $\pr_1(\br^n)$ by
\begin{equation*}
	d_1(m_1,m_2):=\inf_{\gamma\in\Pi(m_1,m_2)}\int_{\br^{2n}} |x-y|d\gamma(x,y),\quad m_1,m_2\in\pr_1(\br^n),
\end{equation*}
where $\Gamma(m_1,m_2)$ denotes the collection of all probability measures on $\br^{2n}$ with marginals $m_1$ and $m_2$. The space $(\pr_1(\br^n),d_1)$ is a complete separable metric space \cite{book_mfg}. Let $\pr_2(\br^n)$ the space of all probability measures $m\in\pr(\br^n)$ such that 
\begin{equation*}
	|m|_{\pr_2(\br^n)}:=\Big(\int_{\br^n} |x|^2m(dx)\Big)^{\frac{1}{2}}<\infty. 
\end{equation*}
We denote by $C^{0}([0,T],\pr_1(\br^n))$ the set of all distribution flows $\mu$ such that $[0,T]\ni t\to \mu(t)\in\pr_1(\br^n)$ is continuous. For $\alpha\in(0,1)$, we denote by 
\begin{align*}
	C^{\alpha}([0,T],\pr_1(\br^n)):=\{\mu\in C^0([0,T],\p_1(\br^n)):\sup_{s\neq t}\frac{d_1(\mu(s),\mu(t))}{|t-s|^{\alpha}}<+\infty\}.
\end{align*}
Let $O$ be a domain in $\br^n$ and $Q_T$ be the cylinder $(0,T)\times O$. For $l>0$, let $C^{\frac{l}{2},l}(\bar{Q}_T)$ denote the set of all functions on $\bar{Q}_T$ having all the continuous derivatives $D^r_tD^s_x$ with $2r+s<l$ and having a finite norm
\begin{equation*}
	|u|_{Q_T}^{(l)}=\sum_{2r+s\le [l]}\sup_{Q_T}|D^r_tD^s_xu|+\sum_{0<l-2r-s<2}\langle D_t^rD_x^su \rangle_{t,Q_T}^{(\frac{l-2r-s}{2})}+\sum_{2r+s=[l]}\langle D_t^rD_x^su \rangle_{x,Q_T}^{(l-[l])},
\end{equation*}
where
\begin{align*}
	&\langle u\rangle_{t,Q_T}^{(\alpha)}=\sup_{(t,x),(t',x)\in\bar{Q}_T}\frac{|u(t,x)-u(t',x)|}{|t-t'|^{\alpha}},\quad 0<\alpha<1,\\
	&\langle u\rangle_{x,Q_T}^{(\alpha)}=\sup_{(t,x),(t,x')\in\bar{Q}_T}\frac{|u(t,x)-u(t,x')|}{|x-x'|^{\alpha}},\quad 0<\alpha<1.
\end{align*}
$C^{\frac{l}{2},l}(\bar{Q}_T)$ is a Banach space \cite{LO}. We also denote by $C^l(\bar{O})$ the space of all time-invariant fields $u\in C^{\frac{l}{2},l}(\bar{Q}_T)$. 

\section{Solubility of fully coupled quasilinear  FBPDEs}\label{MFE}
In this section, we investigate the fully coupled second order quasilinear FBPDEs \eqref{In:3}. Our aim is to prove the existence and uniqueness of solutions of FBPDEs \eqref{In:3} under appropriate assumptions. 

\subsection{Existence}
We first state our assumptions in this subsection. For notational convenience, we use the same constant $L$ for all the conditions below.

\textbf{(A1)} There exists $\gamma\in(0,+\infty)$, such that
\begin{equation*}
	\sum_{i,j=1}^n a_{ij}(t,x,m)\xi_i\xi_j\geq \gamma |\xi|^2, \quad \forall\xi\in\br^n,\quad(t,x,m)\in [0,T]\times\br^n\times \p_1(\br^n).
\end{equation*}
The function $a(t,\cdot,m)\in C^{1+\frac{1}{2}}(\br^n)$ for all $(t,m)\in[0,T]\times\pr_1(\br^n)$, with the function and derivatives being bounded by $L$, Lipschitz continuous in $m\in \p_1(\br^n)$ and $\frac{1}{2}$-Hölder continuous in $t\in[0,T]$, and
\begin{align*}
	|a(t,\cdot,m)|^{(1+\frac{1}{2})}_{\br^n}\le L,\quad (t,m)\in [0,T]\times\pr_1(\br^n).
\end{align*}

\textbf{(A2)} The function $g$ is Lipschitz continuous in $(x,m)\in\br^n\times \p_1(\br^n)$. For all $m\in\pr_1(\br^n)$, $g(\cdot,m)\in C^{2+\frac{1}{2}}(\br^n)$. And there exists $0<\beta<1$ such that 
\begin{equation*}
	|g(\cdot,m)|_{\br^n}^{(1+\beta)}\le L,\quad m\in \pr_1(\br^n).
\end{equation*}

\textbf{(A3)} The initial distribution $m_0\in\pr_2(\br^n)$ is absolutely continuous with respect to the Lebesgue measure, with a density (still denoted by $m_0$) in $C^{\frac{1}{2}}(\br^n)$.

\textbf{(A4)} The function $\hr(t,x,m,\cdot):\br^n\to \br$ is differentiable for all $(t,x,m)\in[0,T]\times\br^n\times\pr_1(\br^n)$. Moreover,
\begin{align*}
	|\hr(t,x,m,p)|\le L(1+|p|^2),\quad (t,x,m,p)\in[0,T]\times\br^n\times\pr_1(\br^n)\times\br^n.
\end{align*}
For $N\in(0,+\infty)$ and $(t,x,m,p),\ (t',x',m',p')\in[0,T]\times\br^n\times\pr_1(\br^n)\times B(0,N)$, where $B(0,N):=\{p\in\br^n:|p|\le N\}$,
\begin{align*}
	&|\frac{\partial \hr}{\partial p}(t,x,m,p)|\le L_N,\\
	&|(\hr,\frac{\partial \hr}{\partial p})(t',x',m',p')-(\hr,\frac{\partial \hr}{\partial p})(t,x,m,p)|\le L_N(|t'-t|^\frac{1}{2}+|x'-x|+d_1(m',m)+|p'-p|),
\end{align*}
for some constant $L_N$ depending on $N$.

We have the following main result of this section:
\begin{theorem}\label{thm:exist}
	Let Assumptions (A1)-(A4) be satisfied. Then, FBPDEs \eqref{In:3} has a solution $(u,m)\in C^{1+\frac{1}{4},2+\frac{1}{2}}([0,T]\times\br^n)\times C^{\frac{1}{2}}([0,T],\p_1(\br^n))$.
\end{theorem}

\begin{remark}
	The boundedness assumptions in $x$ of coefficients hold naturally when FBPDEs \eqref{In:3} are defined on a torus. The proof of Theorem~\ref{thm:exist} relies on application of the Schauder fixed point theorem. We also refer to \cite{VN1} for an alternative application of the Schauder theorem for nonlinear Markov processes. 
\end{remark}

In the rest of this subsection, we give the proof of Theorem~\ref{thm:exist}. The proof is long and is divided into several parts. Following the proof of \cite[Theorem 3.1, p.13]{PC}, we define the set $\dd\subseteq C^0([0,T],\p_1(\br^n))$ as follows:
\begin{equation*}
	\begin{split}
		\mathcal{D}_{C_1}:=\{\mu\in C^0([0,T],\p_1(\br^n)):&\sup_{s\neq t}\frac{d_1(\mu(s),\mu(t))}{|t-s|^{\frac{1}{2}}}\le C_1,\  \sup_{t\in[0,T]}|\mu(t)|^2_{\pr_2(\br^n)}\le C_1\},
	\end{split}
\end{equation*}
where $C_1\in(0,+\infty)$ is waiting to be determined. $\dd$ is a convex closed subset of $C^0([0,T],\p_1(\br^n))$, and is actually compact, due to Arzelà–Ascoli theorem and the fact that the set of probability measures $\mu$ for which $|\mu|^2_{\pr_2(\br^n)}\le C_1$ is compact in $\p_1(\br^n)$ (see \cite[Lemma 5.7]{PC}). We define a map $\Phi:\dd\to\dd$ as follows: for any $\mu\in\dd$, let $u$ be the classical solution to the PDE
\begin{equation}\label{main1}
	\left\{
	\begin{aligned}
		&\frac{\partial u}{\partial t}(t,x)+\sum_{i,j=1}^na_{ij}(t,x,\mu(t,\cdot))\frac{\partial^2u}{\partial x_i\partial x_j}(t,x)\\
		&\quad+\hr(t,x,\mu(t,\cdot),Du(t,x))=0,\quad (t,x)\in[0,T)\times\br^n;\\
		&u(T,x)=g(x,\mu(T,\cdot)),\quad x\in\br^n,
	\end{aligned}
	\right.
\end{equation}
and set $\Phi(\mu)=m$ as the weak solution of the FP equation
\begin{equation}\label{main2}
	\left\{
	\begin{aligned}
		&\frac{\partial m}{\partial t}(t,x)-\sum_{i,j=1}^n\frac{\partial^2}{\partial x_i\partial x_j}(a_{ij}(t,x,\mu(t,\cdot))m(t,x))\\
		&\quad+div[\frac{\partial\hr}{\partial p}(t,x,\mu(t,\cdot),Du(t,x))m(t,x)]=0,\quad (t,x)\in(0,T]\times\br^n;\\
		&m(0,x)=m_0(x),\quad x\in\br^n.
	\end{aligned}
	\right.
\end{equation}
We first show that $\Phi$ is well-defined, that is, for any $\mu\in\dd$, equation \eqref{main1} has a unique solution in $C^{1+\frac{1}{4},2+\frac{1}{2}}([0,T]\times\br^n)$ and equation \eqref{main2} has a unique solution in $\dd$. We then show the continuity of $\Phi$, and Theorem~\ref{thm:exist} as a consequence of the Schauder fixed point theorem.

\subsubsection{Solution of PDE \eqref{main1} in $C^{1+\frac{1}{4},2+\frac{1}{2}}([0,T]\times\br^n)$}
In this subsubsection, we show the solvability of equation \eqref{main1} and give an estimate of the gradient of the solution. We first give the existence and uniqueness result.
\begin{lemma}\label{thm:8.1}
	Let Assumptions (A1), (A2) and (A4) be satisfied. Then, PDE \eqref{main1} has a unique solution $u\in C^{1+\frac{1}{4},2+\frac{1}{2}}([0,T]\times\br^n)$. Moreover, $u$ is bounded by a constant depending only on $(L,T)$.
\end{lemma}

\begin{proof}
	The first equation of PDE \eqref{main1} can be written as the form of equations with principal part in divergence form, namely
	\begin{align*}
		&\frac{\partial u}{\partial t}(t,x)+\sum_{i}^n\frac{\partial}{\partial x_i}[\sum_{j=1}^{n}a_{ij}(t,x,\mu(t,\cdot))\frac{\partial u}{\partial x_j}(t,x)]\\
		&\quad+[\hr(t,x,\mu(t,\cdot),Du(t,x))-\sum_{i,j=1}^{n}\frac{\partial a_{ij}}{\partial x_i}(t,x,\mu(t,\cdot))\frac{\partial u}{\partial x_j}(t,x)]=0,\quad (t,x)\in[0,T)\times\br^n.
	\end{align*}
    From Assumptions (A1), (A2) and (A4), the coefficients of PDE \eqref{main1} have the following properties.
    \begin{enumerate}[(i)]
    	\item The functions $\bar{a}_{ij}(t,x):=a_{ij}(t,x,\mu(t))$ and derivatives $\frac{\partial \bar{a}_{ij}}{\partial x_i}(t,x)$ are Hölder continuous in $(t,x)$ with exponents $(\frac{1}{4},\frac{1}{2})$; and for $N\in(0,+\infty)$, the function $K(t,x,p):=\hr(t,x,\mu(t),p)$ is Hölder continuous in $(t,x,p)$ with exponents $(\frac{1}{4},\frac{1}{2},\frac{1}{2})$ for $(t,x,p)\in[0,T]\times\br^n\times B(0,N)$.
    	
    	Actually, for any $(t',x'),(t,x)\in[0,T]\times\br^n$ and $|p'|,|p|\le N$, from Assumption (A4), we have
    	\begin{align*}
    		|K(t',x',p')-K(t,x,p)|&\le C(N)(|t'-t|^{\frac{1}{2}}+|x'-x|+d_1(\mu(t'),\mu(t))+|p'-p|)\\
    		&\le C(N,C_1) (|t'-t|^{\frac{1}{2}}+|x'-x|+|p'-p|).
    	\end{align*}
    	Here, we use the fact that $d_1(\mu(t'),\mu(t))\le C_1|t'-t|^{\frac{1}{2}}$ since $\mu\in\dd$, and the notation $C(N,C_1)$ stands for a constant depending only on $(N,C_1)$. Then from the boundedness of $K$ when $|p|\le N$, we know that $K$ is Hölder continuous in $(t,x,p)$ with exponents $(\frac{1}{4},\frac{1}{2},\frac{1}{2})$. The proof of the Hölder continuity of functions $(\bar{a}_{ij},\frac{\partial \bar{a}_{ij}}{\partial x_i})(t,x)$ is similar.
    	
    	\item The functions $(\bar{a}_{ij},\frac{\partial \bar{a}_{ij}}{\partial x_i})$ are bounded and there exists a constant $C>0$ such that 
    	\begin{align*}
    		&|K(t,x,p)|\le C (1+|p|)^2,\quad (t,x,p)\in[0,T]\times\br^n\times\br^n.
    	\end{align*}
    
        This is a direct consequence of Assumptions (A1) and (A4), with the constant $C$ depending only on $L$. 
    
    	\item For $N\in(0,+\infty)$, the derivative $\frac{\partial K}{\partial p}(t,x,p)$ is bounded in $[0,T]\times\br^n\times B(0,N)$.
    	
    	Actually, since $\frac{\partial K}{\partial p}(t,x,p)=\frac{\partial \hr}{\partial p}(t,x,\mu(t),p)$, this is a direct consequence of Assumption (A4). 
    \end{enumerate}
	 Then, from the existence and uniqueness theorem of the Cauchy problem of quasi-linear parabolic PDEs \cite[Theorem 8.1, p.495]{LO}, we know that PDE \eqref{main1} has a unique solution $u\in C^{1+\frac{1}{4},2+\frac{1}{2}}([0,T]\times\br^n)$. And in view of \cite[Theorem 2.9, p.23]{LO}, the solution $u$ is bounded by a constant depending only on $(L,T)$.
\end{proof}

We also have the following estimate of the gradient $Du$, which will be used in the following subsections.

\begin{lemma}\label{thm:3.1}
	Let Assumptions (A1), (A2) and (A4) be satisfied and  $u\in C^{1+\frac{1}{4},2+\frac{1}{2}}([0,T]\times\br^n)$ be the solution of PDE \eqref{main1}. Then, there exists $\alpha>0$ depending only on $(n,\gamma,L,\beta)$, such that $|Du|_{[0,T]\times\br^n}^{(\alpha)}$ is bounded by a constant depending only on $(n,\gamma,L,T,\beta)$.
\end{lemma}

\begin{proof}
	Recall \cite[Theorem 3.1, p.437]{LO} on estimates of $Du$ on a bounded domains of $[0,T]\times\br^n$. We choose
	\begin{align*}
		&O_1\subset\subset O_2\subset\subset\dots\subset\subset O_k\subset\subset\dots,\quad \bigcup_{k=1}^{\infty}O_k=\br^n,
	\end{align*}
	and set $Q_T^k:=[0,T]\times O_k$ and $S_T^k:=[0,T]\times\partial O_k$. Then, for any $k\geq 1$, we have
	\begin{align*}
		&u|_{Q_T^k}\in C^{1,2}(Q_T^k),\qquad \sup_{Q_T^k}|u|\le \sup_{[0,T]\times\br^n}|u|\le C(L,T),\\ 
		&|Dg(\cdot,\mu(T))|^{(\beta)}_{O_k}\le |Dg(\cdot,\mu(T))|^{(\beta)}_{\br^n}\le L,
	\end{align*}
     where $C(L,T)$ stands for a constant depending only on $(L,T)$. In view of \cite[Theorem 3.1, p.437]{LO} and our assumptions, there exists a constant $\alpha>0$ depending only on $(n,\gamma,L,\beta)$, such that $|Du|^{\alpha}_{\tilde{Q}^k_T}$ is bounded by a constant depending only on $(n,\gamma,L,T,\beta)$ for any $k\geq 1$, where $\tilde{Q}^k_T\subset Q_T^k$ is separated from the boundary $S^k_T$ by $1$. Since $\bigcup_{k=1}^{\infty}\tilde{Q}_T^k=[0,T]\times\br^n$ and the constants above are independent of $k$, we have $|Du|_{[0,T]\times\br^n}^{(\alpha)}$ is bounded by a constant depending only on $(n,\gamma,L,T,\beta)$. 
\end{proof}

\subsubsection{Solution of PDE \eqref{main2} in $\dd$}\label{belong}
In this subsubsection, we show the existence and uniqueness of the weak solution of equation \eqref{main2}, and show that the weak solution $m$ of FP equation \eqref{main2} belongs to the set $\dd$. Let $(\Omega,\f,\mathbb{P})$ denote a complete filtered probability space augmented by all the $\mathbb{P}$-null sets on which an $n$-dimensional Brownian motion $\{W_t,\ 0\le t\le T\}$ is defined. Consider the following SDE:
\begin{equation}\label{sde_u_mu}
	\left\{
	\begin{aligned}
		&dX_t=\frac{\partial\hr}{\partial p}(t,X_t,\mu(t,\cdot),Du(t,X_t))dt+\sigma(t,X_t,\mu(t,\cdot))dW_t,\quad t\in(0,T];\\
		&X_0=\xi_0\sim m_0.
	\end{aligned}
	\right.
\end{equation}
From Assumptions (A1), (A3) and (A4), the fact that $Du\in C^{\frac{3}{4},\frac{3}{2}}([0,T]\times\br^n)$ and standard arguments of SDE, we know that SDE \eqref{sde_u_mu} has a unique solution $X\in\mathcal{S}^2_{\f}(0,T)$. We set ${m}(t):=\lr(X_t)$ for $t\in[0,T]$. The following definition is borrowed from \cite[Definition 3.2, p.11]{PC}.
\begin{definition}[Weak solution of \eqref{main2}]
	We say that m is a weak solution of PDE \eqref{main2}, if $m\in L^1([0,T],\p_1(\br^n))$ such that for any test function $\varphi\in C_c^{\infty}([0,T)\times\br^n)$, we have
	\begin{equation*}
		\begin{split}
			\int_{\br^n}\varphi(0,x)m_0(dx)+&\int_0^T\int_{\br^n}\Big[\frac{\partial \varphi}{\partial t}(t,x)+\sum_{i,j=1}^na_{ij}(t,x,\mu(t,\cdot))\frac{\partial^2 \varphi}{\partial x_i \partial x_j}(t,x)\\
			&+\sum_{i=1}^n\frac{\partial \hr}{\partial p_i}(t,x,\mu(t,\cdot),Du(t,x))\frac{\partial \varphi}{\partial x_i}(t,x)\Big]m(t,dx)dt=0.
		\end{split}
	\end{equation*}
\end{definition}

The next lemma shows that ${m}$ is a weak solution of PDE \eqref{main2}.

\begin{lemma}\label{lem:weak}
	Let Assumptions (A1)-(A4) be satisfied. Then, ${m}$ is a weak solution of PDE \eqref{main2}.
\end{lemma}
\begin{proof}
	For a test function $\varphi\in C_c^{\infty}([0,T)\times\br^n)$, by applying Itô's formula, we have
	\begin{align*}
		\varphi(t,X_t)=&\varphi(0,\xi_0)+\int_0^t\Big[ \frac{\partial\varphi}{\partial s}(s,X_s)+\frac{1}{2}\tr[\sigma\sigma^T(s,X_s,\mu(s,\cdot))D^2\varphi(s,X_s)]\\
		&+\langle D\varphi(s,X_s),\frac{\partial \hr}{\partial p}(s,X_s,\mu(s,\cdot),Du(s,X_s)) \rangle \Big]ds\\
		&+\int_0^t\langle D\varphi(s,X_s),\sigma(s,X_s,\mu(s,\cdot))dW_s \rangle,\quad t\in[0,T].
	\end{align*}
	Taking the expectation on both sides and noting that $\varphi\in C_c^{\infty}([0,T)\times\br^n)$, from the definition of ${m}$, we know that ${m}$ is a weak solution of PDE \eqref{main2}.
\end{proof}

The following uniqueness result shows that $m$ is the unique weak solution of PDE \eqref{main2}. 
\begin{lemma}\label{pp:weak}
	Let Assumptions (A1)-(A4) be satisfied. Then, PDE \eqref{main2} has at most one weak solution. 
\end{lemma}

\begin{proof}
	We first show that the coefficients of PDE \eqref{main2} belong to the class $C^{\frac{1}{4},\frac{1}{2}}([0,T]\times\br^n)$. Actually, From Assumption (A1) and the fact that $\mu\in\dd$, we know that for $1\le i,j\le n$,
	\begin{equation*}
		a_{ij}(t,x,\mu(t)):=\bar{a}_{ij}(t,x)\in C^{\frac{1}{4},\frac{1}{2}}([0,T]\times\br^n).
	\end{equation*}
	From Lemma~\ref{thm:3.1}, we know that
	\begin{equation}\label{1.5}
		Du\in C^{\frac{3}{4},\frac{3}{2}}([0,T]\times\br^n).
	\end{equation}
	From \eqref{1.5}, Assumption (A4) and the fact that $\mu\in\dd$, we know that for $1\le i\le n$,
	\begin{align*}
		\frac{\partial \hr}{\partial p_i}(t,x,\mu(t),Du(t,x)):=h_i(t,x)\in C^{\frac{1}{4},\frac{1}{2}}([0,T]\times\br^n).
	\end{align*}
    Then, Lemma~\ref{pp:weak} is a direct consequence of \cite[p.12]{PC}, Assumption (A3) and the fact that the coefficients of PDE \eqref{main2} belong to class $C^{\frac{1}{4},\frac{1}{2}}([0,T]\times\br^n)$.
\end{proof}

Now we can prove that the solution of PDE \eqref{main2} belongs to $\dd$.

\begin{lemma}\label{thm:step3}
	Let Assumptions (A1)-(A4) be satisfied. There is $C_1\in(0,+\infty)$ depending only on $(n,\gamma,L,\beta,T,m_0)$, such that the solution $m$ of PDE \eqref{main2} belongs to $\dd$.
\end{lemma}

\begin{proof}
	For $0\le s<t\le T$, 
	\begin{equation}\label{thm_3_1}
	\begin{split}
		d_1(m(s),m(t))&\le \e|X_s-X_t|\\
		&=\e|\int_s^t\frac{\partial\hr}{\partial p}(\tau,X_{\tau},\mu(\tau,\cdot),Du(\tau,X_{\tau}))d\tau+\int_s^t\sigma(\tau,X_{\tau},\mu(\tau,\cdot))dW_\tau|\\
		&\le |t-s|\sup_{(t,x)\in[0,T]\times\br^n}|\frac{\partial \hr}{\partial p}(t,x,\mu(t,\cdot),Du(t,x))|+L|t-s|^{\frac{1}{2}}.
    \end{split}	
    \end{equation}
	From Lemma~\ref{thm:3.1}, we have 
	\begin{equation*}
		\sup_{(t,x)\in[0,T]\times\br^n}|Du(t,x)|\le C(n,\gamma,L,T,\beta),
	\end{equation*}
    where $C(n,\gamma,L,T,\beta)$ is a constant depending only on $(n,\gamma,L,T,\beta)$. So from Assumption (A4), 
    \begin{equation}\label{thm_3_2}
    	\sup_{(t,x)\in[0,T]\times\br^n}|\frac{\partial \hr}{\partial p}(t,x,\mu(t,\cdot),Du(t,x))|\le C(n,\gamma,L,T,\beta).
    \end{equation}
    Therefore, from \eqref{thm_3_1} and \eqref{thm_3_2}, we have
	\begin{align*}
		\sup_{s\neq t}\frac{d_1(m(s),m(t))}{|t-s|^{\frac{1}{2}}}\le C(n,\gamma,L,T,\beta).
	\end{align*}
	Similarly, for any $t\in[0,T]$,
	\begin{equation}\label{thm_3_3}
	\begin{split}
		\int_{\br^n}|x|^2m(t,dx)&=\e|X_t|^2\\
		&\le C(L,T) \e[|\xi_0|^2+1+\sup_{(t,x)\in[0,T]\times\br^n}|\frac{\partial \hr}{\partial p}(t,x,\mu(t,\cdot),Du(t,x))|^2]\\
		&\le C(n,\gamma,L,T,\beta,m_0).
	\end{split}
	\end{equation}
	From \eqref{thm_3_2} and \eqref{thm_3_3}, we have $m\in\dd$ with $C_1=C(n,\gamma,L,\beta,T,m_0)$ independent of $\mu$. The proof is complete.
\end{proof}

\subsubsection{Continuity of $\Phi$}\label{conti_phi}
Now it is clear that the mapping $\Phi:\dd\to\dd$ is well-defined. In this subsubsection, we prove that it is continuous to complete the proof of Theorem~\ref{thm:exist}. 

\begin{lemma}\label{thm:conti}
	Let Assumptions (A1)-(A4) be satisfied. Then, the mapping $\Phi:\dd\to\dd$ is continuous.
\end{lemma}

\begin{proof}
	For notational convenience, we set
	\begin{equation*}
		\rho(\mu,\mu'):=d_{C^0([0,T],\p_1(\br^n))}(\mu,\mu')=\sup_{0\le t\le T}d_1(\mu(t),\mu'(t)),\quad \mu,\mu'\in\dd.
	\end{equation*}
	Let $\{\mu_k,\ k\geq 1\}\subset\dd$ converge to some $\mu\in\dd$ with respect to the metric $\rho$. Let $(u_k,m_k)$ and $(u,m)$ be the solutions of PDEs \eqref{main1}-\eqref{main2} corresponding to $\mu_k$ and $\mu$, respectively. From Assumptions (A1), (A2) and (A4), we know that the coefficients in \eqref{main1} corresponding to $\mu_k$ uniformly converge to the coefficients corresponding to $\mu$. Then we gets the local uniform convergence of $\{u_k,\ k\geq 1\}$ to $u$ by standard arguments. From Lemma~\ref{thm:3.1}, we know that the gradients $\{Du_k,\ k\geq 1\}$ are uniformly bounded and uniformly Hölder continuous and therefore locally uniformly converges to $Du$. For any converging subsequence $\{m_{k_l},\ l\geq 1\}$ of the relatively compact sequence $\{m_k,\ k\geq 1\}$ (since $\dd$ is compact), we assume that $\{m_{k_l},\ l\geq 1\}$ converge to some $\hat{m}\in\dd$. For any $\varphi\in C_c^{\infty}([0,T)\times\br^n)$, since $m_{k_l}$ is a weak solution of \eqref{main2} corresponding to $(\mu_{k_l},u_{k_l})$, we have
	\begin{equation}\label{4:1}
		\begin{split}
			&\int_{\br^n}\varphi(0,x)m_0(dx)+\int_0^T\int_{\br^n}\Big[\frac{\partial\varphi}{\partial t}(t,x)+\sum_{i,j=1}^na_{ij}(t,x,\mu_{k_l}(t,\cdot))\frac{\partial^2\varphi}{\partial x_i\partial x_j}(t,x)\\
			&+\sum_{i=1}^n\frac{\partial \hr}{\partial p_i}(t,x,\mu_{k_l}(t,\cdot),Du_{k_l}(t,x))\frac{\partial\varphi}{\partial x_i}(t,x)\Big]m_{k_l}(t,dx)dt=0.
		\end{split}
	\end{equation}
	From Assumption (A1) and the facts that $\varphi\in C_c^{\infty}([0,T)\times\br^n)$, we have
	\begin{equation}\label{4:1.1^1}
		\begin{split}
			I_{k_l}^1:&=\Big|\int_0^T\int_{\br^n}\sum_{i,j=1}^n[a_{ij}(t,x,\mu_{k_l}(t,\cdot))-a_{ij}(t,x,\mu(t,\cdot))]\frac{\partial^2\varphi}{\partial x_i\partial x_j}(t,x)m_{k_l}(t,dx)dt\Big|\\
			&\le C(L)\rho(\mu_{k_l},\mu)\int_0^T\int_{\br^n} \Big(\sum_{i,j=1}^n|\frac{\partial^2\varphi}{\partial x_i\partial x_j}(t,x)|\Big)m_{k_l}(t,dx)dt.
		\end{split}
	\end{equation}
	From Lemma~\ref{thm:3.1}, we know that there is a constant $C(n,\gamma,L,T,\beta)$ depending only on $(n,\gamma,L,T,\beta)$, such that 
	\begin{equation*}
		\sup_{(t,x)\in[0,T]\times\br^n}|Du(t,x)|\le C(n,\gamma,L,T,\beta),\quad \sup_{(t,x)\in[0,T]\times\br^n}|Du_{k_l}(t,x)|\le C(n,\gamma,L,T,\beta),\quad l\geq 1.
	\end{equation*}
	Therefore, from Assumption (A4) and the fact that $\varphi\in C_c^{\infty}([0,T)\times\br^n)$,
	\begin{equation}\label{4:1.1^2}
		\begin{split}
			I^2_{k_l}:&=\Big|\int_0^T\int_{\br^n}\sum_{i=1}^n[\frac{\partial \hr}{\partial p_i}(t,x,\mu_{k_l}(t,\cdot),Du_{k_l}(t,x))\\
			&\qquad-\frac{\partial \hr}{\partial p_i}(t,x,\mu(t,\cdot),Du(t,x))]\frac{\partial\varphi}{\partial x_i}(t,x)m_{k_l}(t,dx)dt\Big|\\
			&\le C(n,\gamma,L,T,\beta)\int_0^T\int_{\br^n} \Big(\sum_{i=i}^{n}|\frac{\partial\varphi}{\partial x_i}(t,x)|\Big)m_{k_l}(t,dx)dt\\
			&\qquad \cdot [\rho(\mu_{k_l},\mu)+\sup_{(t,x)\in \overline{supp(\varphi)}}|Du_{n_k}(t,x)-Du(t,x)|].
		\end{split}
	\end{equation}
	From Kantorovich-Rubinstein Theorem \cite[Theorem 5.5, p.36]{PC} and the fact that $\{m_{k_l},\ l\geq 1\}$ converge to $\hat{m}$ with respect to $\rho$,  we have for $1\le i,j\le n$,
	\begin{equation}\label{4:1.2}
		\begin{aligned}
			&\lim_{l\to+\infty}\int_0^T\int_{\br^n} |\frac{\partial^2\varphi}{\partial x_i\partial x_j}(t,x)|m_{k_l}(t,dx)dt=\int_0^T\int_{\br^n} |\frac{\partial^2\varphi}{\partial x_i\partial x_j}(t,x)|\hat{m}(t,dx)dt,\\
			&\lim_{l\to+\infty}\int_0^T\int_{\br^n}|\frac{\partial\varphi}{\partial x_i}(t,x))| m_{k_l}(t,dx)dt= \int_0^T\int_{\br^n}|\frac{\partial\varphi}{\partial x_i}(t,x))| \hat{m}(t,dx)dt.
		\end{aligned}
	\end{equation}
	Since $\{\mu_k,\ k\geq 1\}$ converge to $\mu$ with respect to $\rho$ and $\{Du_k,\ k\geq 1\}$ locally uniformly converges to $Du$, we have
	\begin{equation}\label{4:1.3}
		\begin{aligned}
			&\lim_{l\to+\infty}\rho(\mu_{k_l},\mu)= 0;\\
			&\lim_{l\to+\infty}\sup_{(t,x)\in \overline{supp(\varphi)}}|Du_{k_l}(t,x)-Du(t,x)|=0.
		\end{aligned}
	\end{equation}
	Plugging \eqref{4:1.2}-\eqref{4:1.3} into \eqref{4:1.1^1}-\eqref{4:1.1^2}, we have
	\begin{equation}\label{4:2}
		\begin{split}
			\lim_{l\to+\infty}I^1_{k_l}=\lim_{l\to+\infty}I^2_{k_l}= 0.
		\end{split}
	\end{equation}
	Plugging \eqref{4:2} into \eqref{4:1}, we have
	\begin{equation}\label{4:3}
		\begin{split}
			&\int_{\br^n}\varphi(0,x)m_0(dx)+\lim_{l\to+\infty}\int_0^T\int_{\br^n}\Big[\frac{\partial\varphi}{\partial t}(t,x)+\sum_{i,j=1}^na_{ij}(t,x,\mu(t,\cdot))\frac{\partial^2\varphi}{\partial x_i\partial x_j}(t,x)\\
			&+\sum_{i=1}^n\frac{\partial \hr}{\partial p_i}(t,x,\mu(t,\cdot),Du(t,x))\frac{\partial\varphi}{\partial x_i}(t,x)\Big]m_{k_l}(t,dx)dt= 0.
		\end{split}
	\end{equation}
	Again from Kantorovich-Rubinstein Theorem, we have 
	\begin{equation}\label{4:4}
		\begin{split}
			&\lim_{l\to+\infty}\int_0^T\int_{\br^n}\Big[\frac{\partial\varphi}{\partial t}(t,x)+\sum_{i,j=1}^na_{ij}(t,x,\mu(t,\cdot))\frac{\partial^2\varphi}{\partial x_i\partial x_j}(t,x)\\
			&\qquad+\sum_{i=1}^n\frac{\partial\hr}{\partial p_i}(t,x,\mu(t,\cdot),Du(t,x))\frac{\partial\varphi}{\partial x_i}(t,x)\Big]m_{k_l}(t,dx)dt\\
			&= \int_0^T\int_{\br^n}\Big[\frac{\partial\varphi}{\partial t}(t,x)+\sum_{i,j=1}^na_{ij}(t,x,\mu(t,\cdot))\frac{\partial^2\varphi}{\partial x_i\partial x_j}(t,x)\\
			&\qquad+\sum_{i=1}^n\frac{\partial\hr}{\partial p_i}(t,x,\mu(t,\cdot),Du(t,x))\frac{\partial\varphi}{\partial x_i}(t,x)\Big]\hat{m}(t,dx)dt.
		\end{split}
	\end{equation}
	Plugging \eqref{4:4} into \eqref{4:3}, we have 
	\begin{equation*}
		\begin{split}
			&\int_{\br^n}\varphi(0,x)m_0(dx)+\int_0^T\int_{\br^n}\Big[\frac{\partial\varphi}{\partial t}(t,x)+\sum_{i,j=1}^na_{ij}(t,x,\mu(t))\frac{\partial^2\varphi}{\partial x_i\partial x_j}(t,x)\\
			&+\sum_{i=1}^n\frac{\partial\hr}{\partial p_i}(t,x,\mu(t,\cdot),Du(t,x))\frac{\partial\varphi}{\partial x_i}(t,x)\Big]\hat{m}(t,dx)dt= 0,
		\end{split}
	\end{equation*}
	which means that $\hat{m}$ is a weak solution to \eqref{main2} corresponding to $(\mu,u)$. Then, from Lemma~\ref{pp:weak}, we know that $m=\hat{m}$. Up to now we can see that, any converging subsequence $\{m_{k_l},\ l\geq 1\}$ of the relatively compact sequence $\{m_k,\ k\geq 1\}$ converge to $m$. So we know that $\{m_k,\ k\geq 1\}$ converge to $m$. Thus, $\Phi$ is continuous.
\end{proof}

{\it Proof of Theorem~\ref{thm:exist}:} We conclude by Lemmas~\ref{thm:8.1}, \ref{thm:step3} and \ref{thm:conti} and Schauder fixed point theorem that the continuous map $\Phi$ has a fixed point in $\dd$. This fixed point is a solution of FBPDEs \eqref{In:3}. \endproof

\subsection{Uniqueness}
In this subsection, we give the uniqueness result of FBPDEs \eqref{In:3} when the coefficient $\sigma$ is independent of $m$ and the function $\hr$ is of the form
\begin{align*}
	\hr(t,x,m,p)=\hr^0(t,x,m)+\hr^1(t,x,p),\quad
	(t,x,m,p)\in [0,T]\times\br^n\times\pr_1(\br^n)\times\br^{n},
\end{align*}
with the function $\hr^1$ satisfying for any $(t,x,p_1,p_2)\in[0,T]\times\br^n\times\br^n\times\br^n$,
\begin{align}\label{uni_h1}
	\hr^1(t,x,p_2)-\hr^1(t,x,p_1)\le \langle \frac{\partial\hr^1}{\partial p}(t,x,p_1),p_2-p_1 \rangle.
\end{align}
Moreover, let us assume that, besides Assumptions (A1)-(A4), the following monotonicity conditions hold:
\begin{align*}
	&\int_{\br^n}[\hr^0(t,x,m_2)-\hr^0(t,x,m_1)](m_2-m_1)(dx)>0, \quad t\in[0,T],\  m_1,m_2\in\pr_1(\br^n),\  m_1\neq m_2,\\
	&\int_{\br^n}[g(x,m_2)-g(x,m_1)](m_2-m_1)(dx)\geq0, \quad m_1,m_2\in\pr_1(\br^n).
\end{align*}
The above concavity and monotonicity conditions are also used in Lions \cite{PC} and Porretta \cite{POA} for proving the uniqueness result of MFEs \eqref{In:0}. Now we state our uniqueness result for our FBPDEs \eqref{In:3}.

\begin{theorem}\label{thm:uni}
	Under the above conditions, FBPDEs \eqref{In:3} has no more than one solution in $C^{1+\frac{1}{4},2+\frac{1}{2}}([0,T]\times\br^n)\times C^{\frac{1}{2}}([0,T],\p_1(\br^n))$.
\end{theorem}

\begin{proof}
	Let $(u^1,m^1)$ and $(u^2,m^2)$ be two solutions of FBPDEs \eqref{In:3} in $C^{1+\frac{1}{4},2+\frac{1}{2}}([0,T]\times\br^n)\times C^{\frac{1}{2}}([0,T],\p_1(\br^n))$. We set $(\Delta u,\Delta m):=(u^2-u^1,m^2-m^1)$, then
	\begin{equation}\label{uni_1}
		\left\{
		\begin{aligned}
			&\frac{\partial \Delta u}{\partial t}+\sum_{i,j=1}^na_{ij}(t,x)\frac{\partial^2\Delta u}{\partial x_i\partial x_j}+[H^1(t,x,Du^2)-H^1(t,x,Du^1)]\\
			&\qquad+[\hr^0(t,x,m^2(t,\cdot))-\hr^0(t,x,m^1(t,\cdot))]=0,\quad t\in[0,T),\\
			&\frac{\partial \Delta m}{\partial t}-\sum_{i,j=1}^n\frac{\partial^2}{\partial x_i\partial x_j}[a_{ij}(t,x)\Delta m]\\
			&\qquad+\text{div}[\frac{\partial\hr^1}{\partial p}(t,x,Du^2)m^2-\frac{\partial\hr^1}{\partial p}(t,x,Du^1)m^1]=0,\quad t\in(0,T],\\
			&\Delta m(0,x)=0, \quad \Delta u(T,x)=g(x,m^2(T,\cdot))-g(x,m^1(T,\cdot)).
		\end{aligned}
		\right.
	\end{equation}
    Since
    \begin{align*}
    	\int_{\br^n} \Delta u(T,x)\Delta m(T,dx)=\int_0^T\frac{\partial }{\partial t} \int_{\br^n} \Delta u(t,x)\Delta m(t,dx) dt,
    \end{align*}
    from PDEs \eqref{uni_1}, we have
	\begin{align*}
		&\int [g(x,m^2(T,\cdot))-g(x,m^1(T,\cdot))] \Delta m(T,dx)\\
		&\qquad+\int_0^T\int [\hr^0(t,x,m^2(t,\cdot))-\hr^0(t,x,m^1(t,\cdot))] \Delta m(t,dx) dt\\
		&=\int_0^T\int [\hr^1(t,x,Du^1(t,x))-\hr^1(t,x,Du^2(t,x))]\Delta m(t,dx)\\
		&\qquad +\langle D\Delta u(t,x),\frac{\partial\hr^1}{\partial p}(t,x,Du^2(t,x))m^2(t,dx)-\frac{\partial\hr^1}{\partial p}(t,x,Du^1(t,x))m^1(t,dx) \rangle dt.
	\end{align*}
	From \eqref{uni_h1}, we have 
	\begin{align*}
		&[\hr^1(t,x,Du^1)-\hr^1(t,x,Du^2)]\Delta m+\langle D\Delta u,\frac{\partial\hr^1}{\partial p}(t,x,Du^2)m^2-\frac{\partial\hr^1}{\partial p}(t,x,Du^1)m^1 \rangle\\
		&=m_2[\hr^1(t,x,Du^1)-\hr^1(t,x,Du^2)-\langle\frac{\partial\hr^1}{\partial p}(t,x,Du^2),Du^1-Du^2\rangle]\\
		&\qquad +m_1[\hr^1(t,x,Du^2)-\hr^1(t,x,Du^1)-\langle\frac{\partial\hr^1}{\partial p}(t,x,Du^1),Du^2-Du^1\rangle]\\
		&\le 0.
	\end{align*}
	So from the monotonicity conditions of functions $\hr^0$ and $g$, we have $\Delta m=0$ and, therefore, $\Delta u=0$.
\end{proof}

\begin{remark}
	We refer to a weak monotonicity condition
	\begin{align*}
		\e[(\frac{\partial g}{\partial x}(\xi_1,\lr(\xi_1))-\frac{\partial g}{\partial x}(\xi_2,\lr(\xi_2)))(\xi_1-\xi_2)]\geq 0,
	\end{align*} 
    originating in Ahuja \cite{Ahuja} and was later used in Ahuja et al. \cite{AhujaRY} and Huang and Tang \cite{OUR} to give the solvability of conditional distribution dependent forward-backward stochastic differential equations (FBSDEs) associated with MFGs with common noises under differentiability assumptions in $x$ of the Hamiltonian $H$. 
\end{remark}

\subsection{Regularity}
In this subsection, we improve the regularity of our weak solution $m$, and give the existence result of classical solutions of FBPDE \eqref{In:3}. The following Assumptions (A1'), (A3') and (A4') are the regularity-enhanced version of Assumptions (A1), (A3) and (A4).

\textbf{(A1')} The function $a$ satisfies (A1). Moreover, $a(t,\cdot,m)\in C^{2+\frac{1}{2}}(\br^n)$ for all $(t,m)\in[0,T]\times\pr_1(\br^n)$, with the derivatives being bounded by $L$, Lipschitz continuous in $m\in \p_1(\br^n)$ and $\frac{1}{2}$-Hölder continuous in $t\in[0,T]$, and
\begin{align*}
	|a(t,\cdot,m)|^{(2+\frac{1}{2})}_{\br^n}\le L,\quad (t,m)\in [0,T]\times\pr_1(\br^n).
\end{align*}

\textbf{(A3')} The initial distribution $m_0$ satisfies (A3). Moreover, $m_0\in C^{2+\frac{1}{2}}(\br^n)$.

\textbf{(A4')} The function $\hr$ satisfies (A4). Moreover, $\frac{\partial\hr}{\partial p}(t,\cdot,m,\cdot):\br^n\times\br^n\to \br^n$ is differentiable for all $(t,m)\in[0,T]\times\pr_1(\br^n)$, and for $N\in(0,+\infty)$ and $(t,x,m,p),\ (t',x',m',p')\in[0,T]\times\br^n\times\pr_1(\br^n)\times B(0,N)$, where $B(0,N):=\{p\in\br^n:|p|\le N\}$,
\begin{align*}
	&|(\frac{\partial^2\hr}{\partial x\partial p},\frac{\partial^2 \hr}{\partial p^2})(t,x,m,p)|\le L_N,\\
	&|(\frac{\partial^2\hr}{\partial x\partial p},\frac{\partial^2 \hr}{\partial p^2})(t',x',m',p')-(\frac{\partial^2\hr}{\partial x\partial p},\frac{\partial^2 \hr}{\partial p^2})(t,x,m,p)|\\
	&\qquad \le L_N(|t'-t|^\frac{1}{2}+|x'-x|+d_1(m',m)+|p'-p|),
\end{align*}
for some constant $L_N$ depending on $N$.

We have the following existence result of classical solutions of FBPDE \eqref{In:3}:
\begin{theorem}\label{thm:improve}
	Let Assumptions (A1'), (A2), (A3') and (A4') be satisfied. Then, FBPDEs \eqref{In:3} has a classical solution $(u,m)\in C^{1+\frac{1}{4},2+\frac{1}{2}}([0,T]\times\br^n)\times C^{1+\frac{1}{4},2+\frac{1}{2}}([0,T]\times\br^n)$.
\end{theorem}

\begin{remark}
	Our Theorem~\ref{thm:improve} includes as a special case Lions' existence result of MFEs \cite[Theorem 3.1, p.10]{PC} with $\sigma=\sqrt{2}$ and $\hr(x,m,p)=\frac{1}{2}|p|^2+F(x,m)$. 
\end{remark}

\begin{proof}
	From Theorem~\ref{thm:exist}, we know that FBPDEs \eqref{In:3} has a solution $(u,m)\in C^{1+\frac{1}{4},2+\frac{1}{2}}([0,T]\times\br^n)\times \dd$ for some $C_1\in(0,+\infty)$. We only need to show that $m\in C^{1+\frac{1}{4},2+\frac{1}{2}}([0,T]\times\br^n)$. Actually, in view of Assumptions (A1') and (A4'), we know that $m$ satisfies the following PDE:
	\begin{equation}\label{main2'}
		\left\{
		\begin{aligned}
			&\frac{\partial m}{\partial t}(t,x)-\sum_{i,j=1}^n a_{ij}(t,x,m(t,\cdot))\frac{\partial^2m}{\partial x_i\partial x_j}(t,x)\\
			&\quad+\sum_{i=1}^n[\frac{\partial\hr}{\partial p_i}(t,x,m(t,\cdot),Du(t,x))-\sum_{j=1}^n\frac{\partial a_{ij}}{\partial x_j}(t,x,m(t,\cdot))]\frac{\partial m}{\partial x_i}(t,x)\\
			&\quad+\sum_{i=1}^n\big[-\sum_{j=1}^n\frac{\partial^2 a_{ij}}{\partial x_i\partial x_j}(t,x,m(t,\cdot))+\frac{\partial^2\hr}{\partial x_i\partial p_i}(t,x,m(t,\cdot),Du(t,x))\\
			&\quad +\sum_{j=1}^{n}\frac{\partial^2\hr}{\partial p_j\partial p_i}(t,x,m(t,\cdot),Du(t,x))\frac{\partial^2 u}{\partial x_j\partial x_i}(t,x)\big]m(t,x)=0,\quad (t,x)\in(0,T]\times\br^n;\\
			&m(0,x)=m_0(x),\quad x\in\br^n.
		\end{aligned}
		\right.
	\end{equation}
	We first show that the coefficients of PDE \eqref{main2'} belong to the class $C^{\frac{1}{4},\frac{1}{2}}([0,T]\times\br^n)$. From Assumption (A1') and the fact that $m\in\dd$, we know that for $1\le i,j\le n$,
	\begin{equation*}
		(a_{ij},\frac{\partial a_{ij}}{\partial x_j},\frac{\partial^2 a_{ij}}{\partial x_i\partial x_j})(t,x,m(t)):=\bar{a}_{ij}(t,x)\in C^{\frac{1}{4},\frac{1}{2}}([0,T]\times\br^n).
	\end{equation*}
	From the fact that $u\in C^{1+\frac{1}{4},2+\frac{1}{2}}([0,T]\times\br^n)$, we know that
	\begin{equation}\label{1.5'}
		Du\in C^{\frac{3}{4},\frac{3}{2}}([0,T]\times\br^n),\quad D^2u\in C^{\frac{1}{4},\frac{1}{2}}([0,T]\times\br^n).
	\end{equation}
	From \eqref{1.5'}, Assumption (A4') and the fact that $\mu\in\dd$, we know that for $1\le i,j\le n$,
	\begin{align*}
		&(\frac{\partial \hr}{\partial p_i},\frac{\partial^2 \hr}{\partial x_i\partial p_i},\frac{\partial^2 \hr}{\partial p_j\partial p_i})(t,x,\mu(t),Du(t,x)):=h_{ij}(t,x)\in C^{\frac{1}{4},\frac{1}{2}}([0,T]\times\br^n),\\
		&\frac{\partial^2\hr}{\partial p_j\partial p_i}(t,x,\mu(t),Du(t,x))\frac{\partial^2 u}{\partial x_j \partial x_i}(t,x):=k_{ij}(t,x)\in C^{\frac{1}{4},\frac{1}{2}}([0,T]\times\br^n).
	\end{align*}
	Therefore, from \cite[Theorem 5.1, p.320]{LO} and our Assumptions (A1'), (A2), (A3') and (A4'), we have $m\in C^{1+\frac{1}{4},2+\frac{1}{2}}([0,T]\times\br^n)$. The proof is complete.
\end{proof}

\section{Application to mean field games}\label{MFG}
Let $(\Omega,\f,\mathbb{P})$ denote a complete filtered probability space augmented by all the $\mathbb{P}$-null sets on which an $n$-dimensional Brownian motion $\{W_t,\ 0\le t\le T\}$ is defined. In this section, we investigate the optimal strategy of the MFG:
\begin{equation}\label{control}
	\left\{
	\begin{aligned}
		&\hat{\alpha}\in\argmin_{\alpha}J(\alpha|\hat{m}):=\e\big[\int_0^T f(t,X_t^\alpha,\hat{m}_t,\alpha_t)dt+g(X^\alpha_T,\hat{m}_T)\big];\\
		&X_t^\alpha=\xi_0+\int_0^tb(s,X_s^\alpha,\hat{m}_s,\alpha_s)ds+\int_0^t\sigma(s,X_s^\alpha,\hat{m}_s)dW_s,\quad t\in[0,T];\\
		&\hat{m}_t=\mathcal{L}(X_t^{\hat{\alpha}}),\quad t\in[0,T],\quad \xi_0\sim m_0.
	\end{aligned}
	\right.
\end{equation}
where $\alpha$ is an admissible control, taking values in $\br^{d}$, adapted to the filtration $\f=\{\f_t,\ 0\le t\le T\}$ generated by $(W,\xi_0)$. Here, functions $\sigma$ and $g$ are as defined in previous sections, and
\begin{align*}
	&b:[0,T]\times\br^n\times\pr_1(\br^n)\times\br^{d}\to\br^n,\\
	&f:[0,T]\times\br^n\times\pr_1(\br^n)\times\br^{d}\to\br.
\end{align*}
We define the Hamiltonian $H$ as 
\begin{equation}\label{H}
	\begin{split}
		&H(t,x,m,\alpha,p):=\langle p,b(t,x,m,\alpha)\rangle+f(t,x,m,\alpha),\\
		&\qquad(t,x,m,\alpha,p) \in [0,T]\times\br^n\times \pr_1(\br^n)\times\br^{d}\times\br^n,
	\end{split}
\end{equation}
and the minimizing control function $\phi$ 
\begin{equation}\label{Hphi}
	\begin{split}
		\phi(t,x,m,p):=\argmin_{\alpha\in \br^d}H(t,x,m,\alpha,p), \quad (t,x,m,p) \in [0,T]\times\br^n\times\pr_1(\br^n)\times\br^n,
	\end{split}
\end{equation}
under appropriate conditions.

\subsection{Resolvability of MFEs}
We first use the the results in the last section to investigate the solvability of the following second order MFEs
\begin{equation}\label{pde_bf}
	\left\{
	\begin{aligned}
		&\frac{\partial u}{\partial t}(t,x)+H(t,x,m(t,\cdot),\phi(t,x,m(t,\cdot),Du(t,x)),Du(t,x))\\
		&\qquad+\sum_{i,j=1}^na_{ij}(t,x,m(t,\cdot))\frac{\partial^2u}{\partial x_i\partial x_j}(t,x)=0,\quad (t,x)\in(0,T]\times\br^n;\\
		&\frac{\partial m}{\partial t}(t,x)+\sum_{i=1}^n\frac{\partial}{\partial x_i}[b_i(t,x,m(t,\cdot),\phi(t,x,m(t,\cdot),Du(t,x)))m(t,x)]\\
		&\qquad-\sum_{i,j=1}^n\frac{\partial^2}{\partial x_i\partial x_j}[a_{ij}(t,x,m(t,\cdot))m(t,x)]=0,\quad (t,x)\in[0,T)\times\br^n;\\
		&m(0,x)=m_0(x), \quad u(T,x)=g(x,m(T,\cdot)),\quad x\in\br^n,
	\end{aligned}
	\right.
\end{equation}
under Assumptions (A1)-(A3) and the following assumptions on functions $b$, $f$ and $\phi$.

\textbf{(B1)} The function $b(t,x,m,\cdot):\br^{d}\to\br^n$ is differentiable for all $(t,x,m)\in[0,T]\times\br^n\times\pr_1(\br^n)$. Moreover,
\begin{align*}
	|b(t,x,m,\alpha)|\le L(1+|\alpha|),\quad (t,x,m,\alpha)\in [0,T]\times\br^n\times \pr_1(\br^n)\times\br^{d}.
\end{align*}
For $N\in(0,+\infty)$ and $(t,x,m,\alpha),\ (t',x',m',\alpha')\in[0,T]\times\br^n\times\pr_1(\br^n)\times B(0,N)$, where $B(0,N):=\{\alpha\in\br^d:|\alpha|\le N\}$,
\begin{align*}
	&|b(t',x',m',\alpha')-b(t,x,m,\alpha)|\le L_N(|t'-t|^\frac{1}{2}+|x'-x|+d_1(m',m)+|\alpha'-\alpha|),
\end{align*}
for some constant $L_N$ depending on $N$.

\textbf{(B2)}
The function $f(t,x,m,\cdot):\br^{d}\to\br$ is differentiable for all $(t,x,m)\in[0,T]\times\br^n\times\pr_1(\br^n)$. Moreover,
\begin{align*}
	|f(t,x,m,\alpha)|\le L(1+|\alpha|^2),\quad (t,x,m,\alpha)\in [0,T]\times\br^n\times \pr_1(\br^n)\times\br^{d}.
\end{align*}
For $N\in(0,+\infty)$ and $(t,x,m,\alpha),\ (t',x',m',\alpha')\in[0,T]\times\br^n\times\pr_1(\br^n)\times B(0,N)$, where $B(0,N):=\{\alpha\in\br^d:|\alpha|\le N\}$,
\begin{align*}
	&|f(t',x',m',\alpha')-f(t,x,m,\alpha)|\le L_N(|t'-t|^\frac{1}{2}+|x'-x|+d_1(m',m)+|\alpha'-\alpha|),
\end{align*}
for some constant $L_N$ depending on $N$.

\textbf{(B3)} For each $(t,x,m,p)\in[0,T]\times\br^n\times\pr_1(\br^n)\times\br^n$, there exits a unique vector $\phi(t,x,m,p)\in\br^{d}$ satisfying \eqref{Hphi}. Moreover,
\begin{align*}
	|\phi(t,x,m,\alpha)|\le L(1+|p|),\quad (t,x,m,p) \in [0,T]\times\br^n\times\pr_1(\br^n)\times\br^n.
\end{align*}
For $N\in(0,+\infty)$ and $(t,x,m,p),\ (t',x',m',p')\in[0,T]\times\br^n\times\pr_1(\br^n)\times B(0,N)$, where $B(0,N):=\{p\in\br^n:|p|\le N\}$,
\begin{align*}
	&|\phi(t',x',m',p')-\phi(t,x,m,p)|\le L_N(|t'-t|^\frac{1}{2}+|x'-x|+d_1(m',m)+|p'-p|),
\end{align*}
for some constant $L_N$ depending on $N$.

Actually, MFEs \eqref{pde_bf} is the Hamilton system of the optimal control of the non-linear (deterministic) FP equation. The uniform boundedness assumption in the state $x$ of functions $(b,f,g,\phi)$ seems to be restrictive. However,  it is relaxed to include some typical  linear-quadratic cases in Section~\ref{EXA}. Assumption (B3) is also used by Carmona and Delarue in \cite{book_mfg} and by Nourian and Caines in \cite{NM}. Note that Assumption (B3) is satisfied under a convexity condition of $H$ in $\alpha$ and appropriate regularity conditions of $b$ and $f$. Similar convexity assumption can also be found in Porretta \cite{POA}. We now give examples of smooth functions $(\sigma,b,f)$ satisfying Assumptions (A1') and (B1)-(B3).

\begin{example}\label{exp_1}
	We define for $1\le i,j\le n$ and $(x,m,\alpha)\in\br^n\times\pr_1(\br^n)\times\br^{d}$,
	\begin{align*}
		&b_i(x,m,\alpha):=\sin(x_i-\bar{m}_i)+{B^*_i}\alpha,\\ &\sigma_{ij}(x,m):=[2+\cos(x_i-\bar{m}_j)]\delta_{ij},\\ 
		&f(x,m,\alpha):=\frac{|\alpha|^2}{2}+\sum_{j=1}^{n}\sin(x_j-\bar{m}_j),
	\end{align*}
	where $B_i\in\br^{d}$, $\bar{m}_i:=\int_{\br^n} x_im(dx)$ and $\delta_{ij}=1$ when $i=j$ and $\delta_{ij}=0$ when $i\neq j$. In view of Kantorovich-Rubinstein Theorem \cite[Theorem 5.5]{PC}, functions $(\sigma,b,f)$ defined above satisfy Assumptions (A1'), (B1) and (B2). Moreover, $\phi(p)=-{B^{*}}p$ for $p \in \br^n$ where $B^*=(B_1,\cdots B_n)\in\br^{d\times n}$, and Assumption (B3) is satisfied. More generally, our assumptions are satisfied for the following class of functions:
	\begin{align*}
		&b(x,m,\alpha):=B^0(x,m)+{B}\alpha,\quad f(x,m,\alpha):=\frac{|\alpha|^2}{2}+F(x,m),\\
		&(B,\sigma,F):\br^n\times\pr_1(\br^n)\to\br^n\times\br^{n\times n}\times\br,
	\end{align*}
    where {$B\in\br^{n\times d}$,} $\sigma$ is non-degenerate and functions $(B^0,\sigma,\sigma_x,F)$ are bounded and Lipschitz continuous.
\end{example}

In view of FBPDEs \eqref{In:3}, we set
\begin{align}\label{hr_bf}
	\hr(t,x,m,p):=H(t,x,m,\phi(t,x,m,p),p),\quad (t,x,m,p)\in[0,T]\times\br^n\times\pr_1(\br^n)\times\br^n.
\end{align}
From Assumption (B3), we know that for $(t,x,m,p)\in[0,T]\times\br^n\times\pr_1(\br^n)\times\br^n$,
\begin{align*}
	0&=\frac{\partial H}{\partial\alpha}(t,x,m,\alpha,p)|_{\alpha=\phi(t,x,m,p)}\\
	&=\frac{\partial b}{\partial \alpha}(t,x,m,\phi(t,x,m,p))p+\frac{\partial f}{\partial \alpha}(t,x,m,\phi(t,x,m,p)).
\end{align*}
Therefore, 
\begin{align*}
	&\frac{\partial\hr}{\partial p}(t,x,m,p)=b(t,x,m,\phi(t,x,m,p)).
\end{align*}
From Assumptions (B1)-(B3), we know that Assumption (A4) holds for $\hr$ defined in \eqref{hr_bf}. As a consequence of Theorem~\ref{thm:exist}, we have the following existence result for MFEs \eqref{pde_bf}.

\begin{theorem}\label{thm:mfe_exist}
	Let Assumptions (A1)-(A3) and (B1)-(B3) be satisfied. Then, MFEs \eqref{pde_bf} has a solution $(u,m)\in C^{1+\frac{1}{4},2+\frac{1}{2}}([0,T]\times\br^n)\times C^{\frac{1}{2}}([0,T],\p_1(\br^n))$.
\end{theorem}

\begin{remark}
	From Theorem~\ref{thm:improve}, we know that MFEs \eqref{pde_bf} has a classical solution $(u,m)\in C^{1+\frac{1}{4},2+\frac{1}{2}}([0,T]\times\br^n)\times C^{1+\frac{1}{4},2+\frac{1}{2}}([0,T]\times\br^n)$ under Assumptions (A1'), (A2), (A3'), (B1), (B2) and (B3), and when $b$ is differentiable in $x$ and $\phi$ is differentiable in $(x,p)$, and for $N\in(0,+\infty)$ and $(t,x,m,\alpha,p),\ (t',x',m',\alpha',p')\in[0,T]\times\br^n\times\pr_1(\br^n)\times \{\alpha\in\br^d:|\alpha|\le N\}\times\{p\in\br^n:|p|\le N\}$,
	\begin{align*}
		&|(\frac{\partial b}{\partial x},\frac{\partial b}{\partial \alpha})(t,x,m,\alpha)|+|(\frac{\partial \phi}{\partial x},\frac{\partial \phi}{\partial p})(t,x,m,p)|\le L_N,\\
		&|(\frac{\partial b}{\partial x},\frac{\partial b}{\partial \alpha},\frac{\partial \phi}{\partial x},\frac{\partial \phi}{\partial p})(t',x',m',\alpha',p')-(\frac{\partial b}{\partial x},\frac{\partial b}{\partial \alpha},\frac{\partial \phi}{\partial x},\frac{\partial \phi}{\partial p})(t,x,m,\alpha,p)|\\
		&\qquad \le L_N(|t'-t|^\frac{1}{2}+|x'-x|+d_1(m',m)+|\alpha'-\alpha|+|p'-p|),
	\end{align*}
	for some constant $L_N$ depending on $N$. Functions in Example~\ref{exp_1} also satisfy the above conditions.
\end{remark}

\begin{remark}\label{rk:FP}
	As an immediate  consequence of Theorem~\ref{thm:mfe_exist} and Lemmas~\ref{lem:weak} and \ref{pp:weak}, we have the existence result of the following FP equation:
	\begin{equation}\label{FP}
		\left\{
		\begin{aligned}
			&\frac{\partial m}{\partial t}(t,x)-\sum_{i,j=1}^n\frac{\partial^2}{\partial x_i \partial x_j}(a_{ij}(t,x,m(t,\cdot))m(t,x))\\
			&\qquad+div(b(t,x,m(t,\cdot))m(t,x))=0,\quad  (t,x)\in(0,T]\times\br^n;\\
			&m(0,x)=m_0(x),\quad x\in\br^n.
		\end{aligned}
		\right.
	\end{equation}
	Suppose that the function $a$ is non-degenerate, functions $(a,b)$ are bounded, Lipschitz continuous in $(x,m)\in\br^n\times \p_1(\br^n)$ and $\frac{1}{2}$-Hölder continuous in $t\in[0,T]$, and the initial value $m_0$ satisfies Assumption (A3). Then, FP equation \eqref{FP} has a solution $m\in C^{\frac{1}{2}}([0,T],\p_1(\br^n))$, and $m(t)=\lr(X_t)$, where $X\in\mathcal{S}^2_{\f}(0,T)$ is the solution of the McKean-Vlasov SDE
	\begin{equation*}\label{MFSDE}
		\left\{
		\begin{aligned}
			&dX_t=b(t,X_t,\lr(X_t))dt+\sigma(t,X_t,\lr(X_t))dW_t,\quad t\in(0,T];\\
			&X_0=\xi_0\sim m_0.
		\end{aligned}
		\right.
	\end{equation*}
	If moreover function $a$ is twice differentiable in $x$ and  function $b$ is differentiable in $x$, with the derivatives being bounded, Lipschitz continuous in $m\in \p_1(\br^n)$ and $\frac{1}{2}$-Hölder continuous in $(t,x)\in[0,T]\times\br^n$, and the initial value $m_0$ satisfies Assumption (A3'), then $m\in C^{1+\frac{1}{4},2+\frac{1}{2}}([0,T]\times\br^n)$ is a classical solution of FP equation \eqref{FP}.
\end{remark}

Now we give the uniqueness result of MFEs \eqref{pde_bf} when the coefficients $(b,\sigma)$ are independent of $m$ and the function $f$ is of the form
\begin{align}\label{uni_f_sep}
	f(t,x,m,\alpha)=f^0(t,x,m)+f^1(t,x,\alpha),\quad
	(t,x,m,\alpha)\in [0,T]\times\br^n\times\pr_1(\br^n)\times\br^{d}.
\end{align}
Let Assumptions (A1)-(A3) and (B1)-(B3), and  the following monotonicity conditions be satisfied:
\begin{align*}
	&\int[f^0(t,x,m_2)-f^0(t,x,m_1)](m_2-m_1)(dx)>0, \quad t\in[0,T],\  m_1,m_2\in\pr_1(\br^n),\  m_1\neq m_2,\\
	&\int[g(x,m_2)-g(x,m_1)](m_2-m_1)(dx)\geq0, \quad m_1,m_2\in\pr_1(\br^n).
\end{align*}

\begin{theorem}\label{thm:mfe_uni}
	Under the above conditions, MFEs \eqref{pde_bf} has no more than one solution in $C^{1+\frac{1}{4},2+\frac{1}{2}}([0,T]\times\br^n)\times C^{\frac{1}{2}}([0,T],\p_1(\br^n))$.
\end{theorem}

\begin{proof}
	Since the coefficients $(b,\sigma)$ are independent of $m$ and the function $f$ is of the form \eqref{uni_f_sep}, the Hamiltonian $H$ can be divided into two parts
	\begin{equation*}
		H(t,x,m,\alpha,p)=f^0(t,x,m)+H^1(t,x,\alpha,p)
	\end{equation*}
	for $(t,x,m,\alpha,p) \in [0,T]\times\br^n\times \pr_1(\br^n)\times\br^{d}\times\br^n$, where
	\begin{equation*}
		\begin{split}
			&H^1(t,x,\alpha,p):=\langle b(t,x,\alpha),p\rangle+f^1(t,x,\alpha),
		\end{split}
	\end{equation*}
	and the minimizing control function $\phi$ satisfies
	\begin{equation}\label{uni_phi}
		\begin{split}
			\phi(t,x,p):=\argmin_{\alpha\in \br^d}H^1(t,x,\alpha,p), \quad (t,x,p) \in [0,T]\times\br^n\times\br^n.
		\end{split}
	\end{equation}
	We set for $(t,x,m,p)\in[0,T]\times\br^n\times\pr_1(\br^n)\times\br^n$,
	\begin{align*}
		\hr^0(t,x,m):=f^0(t,x,m),\quad \hr^1(t,x,p):=H^1(t,x,\phi(t,x,p),p).
	\end{align*}
    In view of \eqref{uni_phi}, for $(t,x,p)\in [0,T]\times\br^n\times\br^n$,
    \begin{align*}
    	\frac{\partial \hr^1}{\partial p}(t,x,p)=b(t,x,\phi(t,x,p)).
    \end{align*}
    Therefore, from \eqref{uni_phi}, we have for any $(t,x,p_1,p_2)\in[0,T]\times\br^n\times\br^n\times\br^n$,
    \begin{align*}
    	&\hr^1(t,x,p_1)+\langle \frac{\partial\hr^1}{\partial p}(t,x,p_1),p_2-p_1 \rangle\\
    	&=\langle b(t,x,\phi(t,x,p_1)),p_1\rangle+f^1(t,x,\phi(t,x,p_1))+\langle b(t,x,\phi(t,x,p_1)), p_2-p_1 \rangle\\
    	&=H^1(t,x,\phi(t,x,p_1),p_2)\\
    	&\geq H^1(t,x,\phi(t,x,p_2),p_2)=\hr^1(t,x,p_2).
    \end{align*}
    Then, the uniqueness result of MFEs \eqref{pde_bf} follows from Theorem~\ref{thm:uni}.
\end{proof}

\subsection{A verification theorem}\label{APP}
In this subsection, we give a verification theorem for MFG \eqref{control}. Suppose that $(u,m)\in C^{1+\frac{1}{4},2+\frac{1}{2}}([0,T]\times\br^n)\times C^{\frac{1}{2}}([0,T],\pr_1(\br^n))$ is a solution of the MFEs \eqref{pde_bf}. The following theorem shows that the feedback strategy
\begin{align*}
	\bar{\alpha}(t,x):=\phi(t,x,m(t,\cdot),Du(t,x)),\quad (t,x)\in[0,T]\times\br^n
\end{align*}
is optimal for MFG \eqref{control}.

\begin{theorem}[verification theorem]\label{thm_mfg}
	Let Assumptions (A1), (A3), (B1) and (B3) be satisfied and let $(u,m)\in C^{1+\frac{1}{4},2+\frac{1}{2}}([0,T]\times\br^n)\times C^{\frac{1}{2}}([0,T],\pr_1(\br^n))$ be a solution of MFEs \eqref{pde_bf}. Then, $\tilde{\alpha}:=\{\bar{\alpha}(t,\bar{X}_t),\ 0\le t\le T\}$ is an optimal control for MFG \eqref{control}, where $\bar{X}\in\mathcal{S}^2_{\f}(0,T)$ is the solution of SDE
	\begin{equation}\label{sde_app}
		\bar{X}_t=\xi_0+\int_0^t b(s,\bar{X}_s,m(s,\cdot),\bar{\alpha}(s,\bar{X}_s))ds+\int_0^t\sigma(s,\bar{X}_s,m(s,\cdot))dW_s,\quad t\in[0,T].
	\end{equation}
\end{theorem}

\begin{proof}
	Given $(u,m)$, similar as the proof of Lemma~\ref{lem:weak}, we know that $\lr(\bar{X}_t)$ is a weak solution of the following PDE for $\bar{m}$:
	\begin{equation}\label{equ_app}
		\left\{
		\begin{aligned}
			&\frac{\partial \bar{m}}{\partial t}(t,x)-\sum_{i,j=1}^n\frac{\partial^2}{\partial x_i\partial x_j}[a_{ij}(t,x,m(t,\cdot))\bar{m}(t,x)]\\
			&\qquad+\sum_{i=1}^n\frac{\partial}{\partial x_i}[b_i(t,x,m(t,\cdot),\phi(t,x,m(t,\cdot),Du(t,x)))\bar{m}(t,x)]=0,\quad (t,x)\in(0,T]\times\br^n;\\
			&\bar{m}(0,x)=m_0(x),\quad x\in\br^n.
		\end{aligned}
		\right.
	\end{equation}
	From the fact that $(u,m)\in C^{1+\frac{1}{4},2+\frac{1}{2}}([0,T]\times\br^n)\times C^{\frac{1}{2}}([0,T],\pr_1(\br^n))$ and Assumptions (A1), (A3), (B1) and (B3), it is easy to check that the coefficients of PDE \eqref{equ_app} belong to the class $C^{\frac{1}{4},\frac{1}{2}}([0,T]\times\br^n)$ and the initial function belongs to $C^{\frac{1}{2}}(\br^n)$. Similar as Lemma~\ref{pp:weak}, we know that $m$ is the unique weak solution to PDE \eqref{equ_app}. Therefore, $m(t)=\lr(\bar{X}_t)$ for $t\in[0,T]$. 
	
	Let $\alpha$ be an adapted control and $X$ be the corresponding state
	\begin{equation}
        {X}_t=\xi_0+\int_0^tb(s,{X}_s,m(s,\cdot),{\alpha}_s)ds+\int_0^t\sigma(s,{X}_s,m(s,\cdot))dW_s,\quad t\in[0,T].
	\end{equation}
	We have from Itô's formula that
	\begin{align*}
		\e[u(T,X_t)]=&\e\Big[u(0,\xi_0)+\int_0^T \big[\frac{\partial u}{\partial t}(t,X_t)+\langle Du(t,X_t), b(t,X_t,m(t,\cdot),\alpha_t)\rangle\\
		&\quad+\sum_{i,j=1}^na_{ij}(t,X_t,m(t,\cdot))\frac{\partial^2u}{\partial x_i\partial x_j}(t,X_t)\big] dt\Big].
	\end{align*}
	From the terminal condition of $u$ and the definition of the Hamiltonian, we have
	\begin{equation}\label{thm:app_1}
		\begin{split}
			J(\alpha|m)=&\e\Big[u(0,\xi_0)+\int_0^T\big[ \frac{\partial u}{\partial t}(t,X_t)+\sum_{i,j=1}^na_{ij}(t,X_t,m(t,\cdot))\frac{\partial^2u}{\partial x_i\partial x_j}(t,X_t)\\
			&\quad+H(t,X_t,m(t,\cdot),\alpha_t,Du(t,X_t))\big]dt\Big].
		\end{split}
	\end{equation}
	From Assumptions (B3), we know that
	\begin{equation}\label{thm:app_2}
		\begin{split}
			&H(t,X_t,m(t,\cdot),\alpha_t,Du(t,X_t))\\
			&\geq H(t,X_t,m(t,\cdot),\phi(t,X_t,m(t,\cdot),Du(t,X_t)),Du(t,X_t)).
		\end{split}
	\end{equation}
	Plugging \eqref{thm:app_2} into \eqref{thm:app_1}, since $u$ is a solution of \eqref{pde_bf}, we have
	\begin{equation*}\label{thm:app_3}
		\begin{split}
			J(\alpha|m)&\geq \e\Big[u(0,\xi_0)+\int_0^T \big[\frac{\partial u}{\partial t}(t,X_t)+\sum_{i,j=1}^na_{ij}(t,X_t,m(t,\cdot))\frac{\partial^2u}{\partial x_i\partial x_j}(t,X_t)\\
			&\qquad+H(t,X_t,m(t,\cdot),\phi(t,X_t,Du(t,X_t)),Du(t,X_t)) \big]dt\Big]\\
			&= \e[u(0,\xi_0)]=\int_{\br^n}u(0,x)m_0(dx).
		\end{split}
	\end{equation*}
	This shows that $J(\alpha|m)\geq \e[u(0,\xi_0)]$ for any adapted control $\alpha$. If we replace $\alpha$ with $\tilde{\alpha}$ in the above computation, then the state process $X$ becomes $\bar{X}$ and all the above inequalities are equalities. So $J(\tilde{\alpha}|m)=\e[u(0,\xi_0)]$ and the proof is complete.
\end{proof}

As a direct consequence of Theorems~\ref{thm:mfe_exist} and \ref{thm_mfg}, we have the following result.

\begin{corollary}\label{thm_mfg1}
	Let Assumptions (A1)-(A3) and (B1)-(B3) be satisfied. Then, $\tilde{\alpha}:=\{\bar{\alpha}(t,\bar{X}_t),\ 0\le t\le T\}$ is an optimal control for MFG \eqref{control}, where $(u,m)\in C^{1+\frac{1}{4},2+\frac{1}{2}}([0,T]\times\br^n)\times C^{\frac{1}{2}}([0,T],\p_1(\br^n))$ is a solution of MFEs \eqref{pde_bf} and $\bar{X}$ is the solution of SDE \eqref{sde_app}.
\end{corollary}

\begin{remark}
	Carmona and Delarue \cite[Theorem 4.44, p.263]{book_mfg} consider the MFG problem \eqref{control} under the nondegenerate condition with a probabilistic approach. Via the stochastic maximum principle for optimal controls \cite{YH1}, Carmona and Delarue \cite{PA} transform the MFGs into resolvability of distribution dependent FBSDEs. The existence result of the FBSDEs requires a monotonicity condition, which was proposed by Peng and Wu \cite{SP}. Carmona and Delarue prove the existence of the FBSDEs within a linear-convex framework where the volatility $\sigma$ is a constant. Moreover, a weak mean-reverting condition is required. In our work, we use a analytical approach and allow the volatility to depend upon the distribution of the state. The existence of classical solution requires the differentiability of coefficients in $x$, but does not require the monotonicity condition. These differentiability assumptions are expected to be relaxed to more general cases with some appropriate approximation technique. But this would involve further complicated  analysis and go beyond the scope of this paper. Actually, the mean field FBSDEs and mean field FBPDEs are connected to each other. A solution of MFEs \eqref{pde_bf} gives a decoupling function of associated mean field FBSDEs by setting $Y_t=Du(t,X_t)$ and $Z_t=\sigma(t,X_t,m(t,\cdot))D^2u(t,X_t)$, where $(Y,Z)$ are the associated adjoint processes. However, obtaining a solution of FBPDEs from a solution of FBSDEs requires additional regularity conditions of the coefficients. Our work builds a bridge between them.
\end{remark}

\subsection{Discussion on master equation}\label{discuss}
In this subsection, we discuss the relations between MFEs \eqref{pde_bf} and master equation \eqref{master}. The master equation is based on the derivative in $m\in \pr(\br^n)$. Here, we use the concept of linear functional derivative introduced by Carmona and Delarue \cite{CAR}. There are  various studies on the differentiability of functionals of probability distribution. We refer to Lions \cite{PL} for the $L$-derivative for functions defined on the space $\pr_2(\br^n)$. We also refer the reader to Bensoussan et al. \cite{AB4} for a discussion of relations between the differentiability of functions of probability distribution in $\pr_2(\br^n)$ and of probability density in $L^2(\br^n)$. 

Suppose that master equation \eqref{master} has a classical solution $U$, then, it is easy to check that \eqref{decouple} decouples MFEs \eqref{pde_bf}. Now we try to construct a solution of master equation \eqref{master} by using the solution of MFEs \eqref{pde_bf}. Suppose that $(u,m)\in C^{1+\frac{1}{4},2+\frac{1}{2}}([0,T]\times\br^n)\times C^{\frac{1}{2}}([0,T],\pr_1(\br^n))$ is the soluton of MFEs \eqref{pde_bf}. For $t\in[0,T]$ and $\mu$ satisfying Assumption (A3), we denote by $(u^{t,\mu},m^{t,\mu})\in C^{1+\frac{1}{4},2+\frac{1}{2}}([t,T]\times\br^n)\times C^{\frac{1}{2}}([t,T],\pr_1(\br^n))$ the solution of the following MFEs defined on $[t,T]\times\br^n$
\begin{equation*}
	\left\{
	\begin{aligned}
		&\frac{\partial u^{t,\mu}}{\partial t}(s,x)+H(s,x,m^{t,\mu}(s,\cdot),\phi(s,x,m^{t,\mu}(s,\cdot),Du^{t,\mu}(s,x)),D^{t,\mu}u(s,x))\\
		&\qquad+\sum_{i,j=1}^na_{ij}(s,x,m^{t,\mu}(s,\cdot))\frac{\partial^2u^{t,\mu}}{\partial x_i\partial x_j}(s,x)=0,\quad (s,x)\in(t,T]\times\br^n;\\
		&\frac{\partial m^{t,\mu}}{\partial t}(s,x)+\sum_{i=1}^n\frac{\partial}{\partial x_i}[b_i(s,x,m^{t,\mu}(s,\cdot),\phi(s,x,m^{t,\mu}(s,\cdot),Du^{t,\mu}(s,x)))m^{t,\mu}(s,x)]\\
		&\qquad-\sum_{i,j=1}^n\frac{\partial^2}{\partial x_i\partial x_j}[a_{ij}(s,x,m^{t,\mu}(s,\cdot))m^{t,\mu}(s,x)]=0,\quad (s,x)\in[t,T)\times\br^n;\\
		&m^{t,\mu}(t,x)=\mu(x), \quad u^{t,\mu}(T,x)=g(x,m^{t,\mu}(T,\cdot)),\quad x\in\br^n.
	\end{aligned}
	\right.
\end{equation*}
Actually, from the flow property, we have for $t\in[0,T]$,
\begin{align}\label{flow}
	m^{t,m(t)}(s,\cdot)=m(s,\cdot),\quad u^{t,m(t)}(s,x)=u(s,x),\quad (s,x)\in[t,T]\times\br^n.
\end{align}
We define for $(t,x)\in[0,T]\times\br^n$ and $\mu$ satisfying Assumption (A3)
\begin{align}\label{rela}
	U(t,x,\mu):=u^{t,\mu}(t,x),
\end{align}
then, since $u$ is the solution of Bellman equation, we have the following probabilistic interpretation for $U$:
\begin{align*}
	U(t,x,\mu)=\e\Big[&\int_t^T f(s,X_s^{t,x,\mu},m^{t,\mu}(s,\cdot),\phi(s,X_s^{t,x,\mu},m^{t,\mu}(s,\cdot),Du^{t,\mu}(s,X_s^{t,x,\mu})))ds\\
	&+g(X_T^{t,x,\mu},m^{t,\mu}(T,\cdot))\Big],
\end{align*}
where $X^{t,x,\mu}\in\mathcal{S}^2_\f(t,T)$ is the solution of the following SDE
\begin{align*}
	X_s^{t,x,\mu}=&x+\int_t^s b(r,X_r^{t,x,\mu},m^{t,\mu}(r,\cdot),\phi(r,X_r^{t,x,\mu},m^{t,\mu}(r,\cdot),Du^{t,\mu}(r,X_r^{t,x,\mu})))dr\\
	&+\int_t^s \sigma(r,X_r^{t,x,\mu},m^{t,\mu}(r,\cdot))dW_r,\quad s\in[t,T].
\end{align*}
Suppose that $U$ is smooth enough. Then, from the It\^o's formula for $U$ (see \cite{BR, OUR1}), the flow property \eqref{flow} and the relationship \eqref{rela}, we know that $U$ satisfies the master equation \eqref{master}. A rigorous proof requires the differentiability property of the map $\mu\mapsto U(t,x,\mu)$ and relies on a preliminary study of the differential properties of the flow $\mu\mapsto (u^{t,\mu}, m^{t,\mu})$, which is challenging. We refer to our work \cite{OUR1} for the study of the differentiability property of the map $\mu\mapsto U(t,x,\mu)$ and the solvability of master equations of particular types. 

\section{Linear quadratic problems}\label{EXA}
In this section, we relax the boundedness assumption in the state $x$ of functions $(a,b,f,g,\phi)$ to include the linear-quadratic case. We denote by $\mathcal{E}(b,\sigma,f,g,m_0)$ the MFEs \eqref{pde_bf} with coefficients $(b,\sigma,f,g,m_0)$, and denote by $\phi^{b,f}$ the feedback function \eqref{Hphi} corresponding to $(b,f)$. We denote by $P(b,\sigma,f,g,\xi_0)$ the MFG \eqref{control} corresponding to $(b,\sigma,f,g,\xi_0)$. We first give the following corollary of Theorems~\ref{thm:mfe_exist} and \ref{thm_mfg}, which allows the drift $b$ to be linear in $x$.
\begin{corollary}\label{cor1}
	Suppose that functions $(\sigma,f,g)$ are independent of $x$,
	\begin{align*}
		&b(t,x,m,\alpha)=\lambda(t)x+b'(t,m,\alpha),\quad b':[0,T]\times\pr_1(\br^n)\times\br^{d}\to\br^n,
	\end{align*}
	with the function $\lambda:[0,T]\to\br^{n\times n}$ being bounded and continuous and functions $(b',\sigma,f,g,m_0)$ satisfy Assumptions (A1)-(A3) and (B1)-(B3). Then, $\mathcal{E}(b,\sigma,f,g,m_0)$ has a solution $(u,m)\in C^{1,2}([0,T]\times\br^n)\times C^{0}([0,T],\p_1(\br^n))$, and the feedback $\alpha(t,x):=\phi^{b,f}(t,m(t,\cdot),Du(t,x))$ is an optimal control of $P(b,\sigma,f,g,\xi_0)$.
\end{corollary}

\begin{proof}
We use the following transformation:
\begin{align}\label{trans}
	Y_t:=e^{-\int_0^t \lambda(s)ds}X_t,\quad t\in[0,T].
\end{align}
Then, $\mathcal{E}(b,\sigma,f,g,m_0)$ for $(u,m)$ is transformed into $\mathcal{E}(\tilde{b},\tilde{\sigma},{f},{g},m_0)$ for $(\tilde{u},\tilde{m})$, and $P(b,\sigma,f,g,\xi_0)$ is transformed into MFG $P(\tilde{b},\tilde{\sigma},f,g,\xi_0)$, where for $(t,m,\alpha,p)\in[0,T]\times\pr_1(\br^n)\times\br^{d}\times\br^n$,
\begin{align}
	&\tilde{b}(t,m,\alpha)=e^{-\int_0^t \lambda(s)ds}b'(t,m,\alpha),\qquad \tilde{\sigma}(t,m)=e^{-\int_0^t \lambda(s)ds}\sigma(t,m),\notag\\
	&\phi^{\tilde{b},f}(t,m,p)=\phi^{b,f}(t,m,e^{-\int_0^t \lambda(s)ds}p),\label{cor1_0}
\end{align}
and for $(t,y)\in[0,T]\times\br^n$,
\begin{align}
	&\tilde{u}(t,y)=u(t,e^{\int_0^t \lambda(s)ds}y),\qquad \tilde{m}(t,y)=e^{\int_0^t \lambda(s)ds}m(t,e^{\int_0^t \lambda(s)ds}y).\label{cor1_1}
\end{align}
Since functions $(\sigma,b',\phi^{b,f})$ satisfy Assumptions (A1), (B1) and (B3) and function $\lambda$ is bounded and continuous, it is easy to check that functions $(\tilde{\sigma},\tilde{b},\phi^{\tilde{b},f})$ satisfy Assumptions (A1), (B1) and (B3). From Theorem~\ref{thm:mfe_exist}, $\mathcal{E}(\tilde{b},\tilde{\sigma},{f},{g},m_0)$ has at least one solution $(\tilde{u},\tilde{m})\in C^{1+\frac{1}{4},2+\frac{1}{2}}([0,T]\times\br^n)\times C^{\frac{1}{2}}([0,T],\p_1(\br^n))$. From Corollary~\ref{thm_mfg1}, the feedback $\tilde{\alpha}(t,y):=\phi^{\tilde{b},f}(t,\tilde{m}(t,\cdot),D\tilde{u}(t,y))$ is an optimal control of $P(\tilde{b},\tilde{\sigma},f,g,\xi_0)$. In view of \eqref{cor1_0}-\eqref{cor1_1}, $\mathcal{E}(b,\sigma,f,g,m_0)$ has a solution $({u},{m})\in C^{1,2}([0,T]\times\br^n)\times C^{0}([0,T],\p_1(\br^n))$, and the feedback $\alpha(t,x):=\phi^{b,f}(t,m(t,\cdot),Du(t,x))$ is an optimal control of $P(b,\sigma,f,g,\xi_0)$.
\end{proof}

Now we consider the linear-quadratic case.
\begin{corollary}\label{cor_lq}
Consider $P(b,\sigma,f,g,\xi_0)$ with coefficients
\begin{align*}
	&b(x,\alpha)=Ax+B\alpha,\quad \sigma:[0,T]\times\pr_1(\br^n)\to\br^{n\times n},\\
	&f(t,x,m,\alpha)=\frac{1}{2}x^*Qx+\frac{1}{2}\alpha^*R\alpha+F(t,m),\quad g(x,m)=\frac{1}{2}x^*Mx+G(m),
\end{align*}
where $A,B,Q,R,M$ are matrices of suitable sizes, $Q,R,M$ are symmetric and $R$ is positive definite. Suppose that $\sigma$ is non-degenerate, functions $(\sigma,F,G)$ are bounded, Lipschitz continuous in $m\in\pr_1(\br^n)$ and $\frac{1}{2}$-Hölder continuous in $t\in[0,T]$, and the initial value $m_0$ satidfies Assumption (A3). Then, the optimal control is given by the feedback $\alpha(t,x):=-R^{-1}B^*P(t)x$, where $P(\cdot)$ is the solution of the standard Riccati equation \eqref{lq_1}.
\end{corollary}

\begin{proof}
We look for a solution of $\mathcal{E}(b,\sigma,f,g,m_0)$ of the form $u(t,x)=\beta(t)+\frac{1}{2}x^*P(t)x$. Then, the feedback is of the form $\alpha(t,x)=-R^{-1}B^*P(t)x$, and $\mathcal{E}(b,\sigma,f,g,m_0)$ is equivalent to the following equations:
\begin{align}
	&\dot{P}(t)+A^*P(t)+P(t)A-2P(t)BR^{-1}B^*P(t)+Q=0,\quad P(T)=M,\label{lq_1}\\
	&\dot{\beta}(t)+\sum_{i,j=1}^na_{ij}(t,m(t,\cdot))P_{ij}(t)+F(t,m(t,\cdot))=0,\quad \beta(T)=G(m(T,
	\cdot)),\label{lq_2}\\
	&\frac{\partial m}{\partial t}(t,x)+\text{div}[m(t,x)(A-BR^{-1}B^*P(t))x]\notag\\
	&\qquad-\sum_{i,j=1}^na_{ij}(t,m(t,\cdot))\frac{\partial^2m}{\partial x_i\partial x_j}(t,x)=0,\quad m(0,x)=m_0(x). \label{lq_3}
\end{align}
Let $P(\cdot)\in C^1([0,T];\br^{n\times n})$ be the solution of the standard Riccati equation \eqref{lq_1}. By using transformation \eqref{trans} with $\lambda(t):=A-BR^{-1}B^*P(t)$, in view of Remark~\ref{rk:FP}, we know that PDE \eqref{lq_3} has a solution $m\in C^{\frac{1}{2}}([0,T],\p_1(\br^n))$ and $m(t)=\lr(\bar{X}_t)$, where
\begin{equation*}
		\bar{X}_t=\xi_0+\int_0^t(A-BR^{-1}B^*P(s))\bar{X}_sds+\int_0^t\sigma(s,m(s,\cdot))dW_s,\quad t\in[0,T].
\end{equation*}
Then equation \eqref{lq_2} has a classical solution $\beta\in C^1[0,T]$. Similar to the proof of Theorem~\ref{thm_mfg}, we know that the feedback $\alpha(t,x)$ is an optimal control of $P(b,\sigma,f,g,\xi_0)$.
\end{proof}

\begin{remark}
	Bensoussan et al. \cite{AB2,AB3} solve the FBPDEs for linear quadratic mean field games under the condition that the volatility $\sigma$ is a constant. Our work cannot include Bensoussan's results because our drift here is independent of the distribution variable $m$. However, our work also go beyond their framework by allowing the volatility $\sigma$ to depend on $m$. Our boundedness assumption in $m$ is expected to be relaxed by using some appropriate approximation techniques. 
\end{remark}


	
\end{document}